\documentclass[12pt]{amsart}
\usepackage{amscd}
\usepackage{amssymb}
\usepackage{a4wide}
\usepackage{amstext}
\usepackage{amsthm}
\usepackage{mathrsfs}
\usepackage[final]{hyperref}
\hypersetup{unicode= false, colorlinks=true, linkcolor=blue,
anchorcolor=blue, citecolor=green, filecolor=red, menucolor=blue,
pagecolor=blue, urlcolor=blue} 
\numberwithin{equation}{section}

\newcommand{\CC}{\mathbb {C}}

\newcommand{\RR}{\mathbb{R}}
\newcommand{\Z}{\mathbb{Z}}

 \DeclareMathOperator{\dist}{dist}

\renewcommand{\phi}{\varphi}
\newcommand{\rea}{{\rm Re}\,}
\newcommand{\ima}{{\rm Im}\,}
\newcommand{\pw}{\mathcal{P}W_\pi}
\newcommand{\ppw}{\mathcal{P}W_{2\pi}}

\newcommand{\co}{\mathbb{C}}
\newcommand{\na}{\mathbb{N}}
\newcommand{\zl}{\mathbb{Z}}
\newcommand{\ZZ}{\mathbb{Z}}
\newcommand{\cp}{\mathbb{C_+}}
\newcommand{\rl}{\mathbb{R}}

\newcommand{\he}{\mathcal{H}(E)}
\newcommand{\nar}{\mathcal{N}}

\newtheorem{Thm}{Theorem}[section]
\newtheorem{theorem}[Thm]{Theorem}
\newtheorem{lemma}[Thm]{Lemma}
\newtheorem{proposition}[Thm]{Proposition}
\newtheorem{corollary}[Thm]{Corollary}
\newtheorem{remark}[Thm]{Remark}

\begin{document}
\sloppy
\title[Hereditary completeness for systems of exponentials]{Hereditary 
completeness for systems \\ 
of exponentials and reproducing kernels}
\author{Anton Baranov, Yurii Belov, Alexander Borichev}
\address{Anton Baranov, 
\newline
Department of Mathematics and Mechanics,
St.Petersburg State University,\hfill\hfill\\
St.Petersburg, Russia,
\newline {\tt anton.d.baranov@gmail.com}
\newline\newline \phantom{x}\,\, Yurii Belov,
\newline 
Chebyshev Laboratory,
St.Petersburg State University,
St.Petersburg, Russia,
\newline {\tt j\_b\_juri\_belov@mail.ru}
\newline\newline \phantom{x}\,\, Alexander Borichev, 
\newline Laboratoire d'Analyse, Topologie, Probabilit\'es, 
Aix--Marseille Universit\'e,\hfill\hfill\newline Marseille, France,
\newline {\tt borichev@cmi.univ-mrs.fr}
}
\thanks{The first and the second author were supported by the Chebyshev Laboratory 
(St.Petersburg State University) under RF Government grant 11.G34.31.0026.
The first author was partially supported by RFBR grant 11-01-00584-a 
and by Federal Program of Ministry of Education 2010-1.1-111-128-033.
The research of the third author was partially supported by the ANR grants DYNOP and FRAB}

\begin{abstract}
We solve the spectral synthesis problem for exponential systems
on an interval.  Namely, we prove that any complete and minimal  
system of exponentials $\{e^{i\lambda_n t}\}$ in $L^2(-a,a)$
is hereditarily complete up to a one-dimensional defect. 
This means that there is at most one (up to a constant factor) 
function $f$ which is
orthogonal to all the summands in its formal Fourier series 
$\sum_n (f,\tilde e_n) e^{i\lambda_n t}$, where
$\{\tilde e_n\}$ is the system biorthogonal to 
$\{e^{i\lambda_n t}\}$.
However, this one-dimensional defect is possible and, thus, 
there exist nonhereditarily complete exponential systems.
Analogous results are obtained for systems 
of reproducing kernels in de Branges spaces.
For a wide class of de Branges spaces 
we construct nonhereditarily complete systems of reproducing kernels, 
thus answering a question posed by N. Nikolski.
\end{abstract}

\maketitle

\section{Introduction and main results}

\subsection{Hereditary completeness in general setting}

A system of vectors  $\{x_n\}_{n\in N}$ in a separable Hilbert space 
$H$ is said to be {\it exact} if it is both {\it complete} 
(i.e., $\overline{Span}\{x_n\} = H$)
and {\it minimal} 
(i.e., $\overline{Span}\{x_n\}_{n\neq n_0} \neq H$ for any $n_0$). 
Given an exact system we consider its (unique) {\it biorthogonal} 
system $\{x'_n\}_{n\in N}$ which satisfies 
$ (x_m, x'_n) = \delta_{mn}$.
Then to every element $x\in H$ we associate its formal Fourier series
$$
x  \sim \sum_{n\in N}(x,x'_n)x_n.
$$
A natural condition is that this correspondence is one-to-one:
no nonzero vector generates 
zero series, in other words the biorthogonal system $\{x'_n\}$ is also complete. 
Another important property is the possibility to reconstruct  the vector 
$x$ from its Fourier series:
$$
x\in \overline{Span} \, \{(x, x_n') x_n\}. 
$$
If this holds, we say that the system $\{x_n\}_{n\in N}$ is {\it hereditarily complete}.
We will use an equivalent description:
for any partition $N = N_1 \cup N_2$, $N_1 \cap N_2 =\emptyset$,  
the system 
$$
\{x_n\}_{n\in N_1} \cup \{x'_n\}_{n\in N_2}
$$    
is complete in $H$.
In the opposite situation (i.e., when 
$\{x_n\}$ and $\{x'_n\}$ are complete, 
but $\{x_n\}$ is not hereditarily complete) we say
that system is nonhereditarily complete. 

The importance of this notion 
is related to the spectral synthesis problem for linear operators.
If $\{x_n\}$ is the sequence of eigenfunctions
and root functions of some compact operator (with trivial kernel), then the
hereditary completeness of $\{x_n\}$
is equivalent to the possibility of 
the so-called {\it spectral synthesis} for this operator,
i.e., its restriction to any invariant subspace is complete
(see \cite{markus} or \cite[Chapter 4]{hrnik}).

The condition that the biorthogonal system $\{x_n'\}$ is also complete in $H$ is by no means automatic and corresponding examples 
can be easily constructed. It is less trivial to give examples
of the situations where both
$\{x_n\}$ and $\{x'_n\}$ are complete, but the system $\{x_n\}$
fails to be hereditarily complete. In fact, first examples
go back to Hamburger \cite{hamb} who constructed
a compact operator with a complete set of eigenvectors, whose restriction 
to an invariant subspace is a nonzero Volterra operator
(and, hence, is not complete). 
Further examples of nonhereditarily complete systems
were found by Markus \cite{markus} and Nikolski \cite{nik69}, while 
a general approach to constructing nonhereditarily
complete systems was developed by Dovbysh, Nikolski and Sudakov 
\cite{dn, dns}. Any nonhereditarily complete system gives 
an example of an exact system which is not a summation basis.
On the other hand, uniform minimality 
and closeness to an orthonormal system  
may be combined with nonhereditary completeness \cite{dns}.

\subsection{Hereditary completeness for exponential systems}
It is natural to study the problem of hereditary completeness
for special systems in functional spaces, e.g. those which appear 
as families of eigenvectors and root vectors of a certain
operator. Exponential systems form an important class in this respect. 
Let $\Lambda = \{\lambda_n\} \subset \co$ and let 
$e_\lambda (t) = \exp(i\lambda t)$. We consider
the exponential system $\{e_\lambda\}_{\lambda\in\Lambda}$ in $L^2(-a, a)$, $a>0$.
It was shown by Young \cite{young} that, in contrast to the general situation,
for any exact system of exponentials its biorthogonal system is complete.
Another approach to this problem was suggested in \cite{gub},
where it is shown that any exact system of exponentials
is the system of eigenfunctions of the differentiation operator
$i\frac{d}{dx}$ in $L^2(-a, a)$ with a certain generalized boundary 
condition.

Applying the Fourier transform $\mathcal{F}$
one reduces the problem for exponential systems
in $L^2(-\pi, \pi)$
to the same problem for systems of reproducing kernels in 
the Paley--Wiener space
$\pw = \mathcal{F} L^2(-\pi,\pi)$. Recall
that the reproducing kernel of $\pw$ corresponding to a point $\lambda\in\co$
is of the form
$$
K_\lambda(z) = \frac{\sin \pi (z-\overline \lambda)}{\pi(z-\overline 
\lambda)}, \qquad f(\lambda) = (f,K_\lambda)_{\pw}.
$$

Hereditary completeness of exponential systems
is a particular case of the following problem posed by Nikolski:
whether there exist nonhereditarily complete
systems of reproducing kernels in the model subspaces of the 
Hardy space (for the theory of model spaces see \cite{nk12};
the Paley--Wiener space and de Branges spaces are such spaces up to a canonical unitary equivalence).
Let us also recall a related result by  
Olevskii \cite{OLE}: there exists an orthonormal basis 
$\{\phi_n\}$ in $L^2(-\pi, \pi)$
consisting of trigonometric polynomials, for which the approximation 
of functions $f$ by the sums $\sum_{n:\, (f, \phi_n) \ne 0}  c_n \phi_n$ fails in the metric of $C[-\pi, \pi]$ or $L^p(-\pi, \pi)$, $p>2$.

We completely solve the problem of hereditary completeness 
for exponential systems. Namely, we show that  the hereditary completeness 
holds up to a possible one-dimensional defect.
 
Let $\Lambda\subset \CC$ be such that the system of reproducing kernels
$\{K_\lambda\}_{\lambda\in \Lambda}$ is exact 
in the Paley--Wiener space $\pw$. 
Then the biorthogonal system $\{g_\lambda\}_{\lambda\in \Lambda}$ 
is given by 
$$
g_\lambda(z) = \frac{G(z)}{G'(\lambda)(z-\lambda)}, 
$$ 
where $G$ is the so-called generating function of the set 
$\Lambda$.
By the above-mentioned result of 
Young,
$\{g_\lambda\}_{\lambda\in \Lambda}$  is also an exact system.
It is well known that $G$ is a function of exponential 
type $\pi$ and has only simple zeros at the points of $\Lambda$.

\begin{theorem}
\label{main1}
If $\{K_\lambda\}_{\lambda\in \Lambda}$ is exact 
in the Paley--Wiener space $\pw$, 
then for any partition $\Lambda = \Lambda_1 \cup \Lambda_2$, the 
orthogonal complement in $\pw$ to the system
\begin{equation}
\label{syst}
\{g_\lambda\}_{\lambda\in \Lambda_1} \cup
\{K_\lambda\}_{\lambda\in \Lambda_2}
\end{equation}
is at most one-dimensional.
\end{theorem}

Moreover, there are certain obstacles for the existence
of this exceptional one-dimensional complement. 
This can not happen when the sequence $\Lambda_1$ has non-zero
upper density. Given a sequence $\Lambda$ set
$$
D_+(\Lambda) = \limsup_{r\to\infty} \frac{n_r(\Lambda)}{2r},
$$
where $n_r(\Lambda)$ is the usual counting function of the
sequence $\Lambda$, $n_r(\Lambda) ={\rm card}\,\{ {\lambda\in\Lambda, |\lambda|\leq r}\}$.

\begin{theorem}
\label{main2}
Let $\Lambda \subset \mathbb C$, let the system $\{K_\lambda\}_{\lambda\in \Lambda}$ 
be exact in $\pw$,  and let the partition $\Lambda = \Lambda_1 \cup \Lambda_2$
satisfy $D_+(\Lambda_1) >0$.
Then the system \eqref{syst} is complete in $\pw$. 
\end{theorem}

Surprisingly, the one-dimensional defect for exponential systems
is still possible.

\begin{theorem}
\label{maincount}
There exist a system of exponentials $\{e^{i\lambda_n t}\}_{n\in \mathbb Z}$, 
$\lambda_n\in \rl$, 
which is complete and minimal in $L^2(-\pi, \pi)$, but
is not hereditarily complete.
\end{theorem}

Thus, hereditary completeness may fail even for exponential systems
(reproducing kernels of the Paley--Wiener space), which answers
the question of Nikolski. Further counterexamples will be discussed in
the next subsection.

\subsection{Reproducing kernels of the de Branges spaces}
The above results may be extended to the de Branges spaces.
Let $E$ be an entire function in the Hermite--Biehler class, that is 
$E$ has no zeros on $\mathbb R$, and 
$$
|E(z)| > |E^*(z)|, \qquad z\in {\mathbb{C}_+},
$$
where $E^* (z) = \overline {E (\overline z)}$. 
With any such function we associate the {\it de Branges space}
$\mathcal{H} (E) $ which consists of all entire functions 
$F$ such that $F/E$ and $F^*/E$ restricted to $\mathbb{C_+}$ belong
to the Hardy space $H^2=H^2(\mathbb{C_+})$. 
The inner product in $\he$ is given by
$$
( F,G)_E = \int_\rl \frac{F(t)\overline{G(t)}}{|E(t)|^2} \,dt.
$$
The reproducing kernel of the de Branges space ${\mathcal H} (E)$
corresponding to the point $w\in \mathbb{C}$ is given by
$$
K_w(z)=\frac{\overline{E(w)} E(z) - \overline{E^*(w)} E^*(z)}
{2\pi i(\overline w-z)} 
.
$$

The Hilbert spaces of entire functions $\he$ were introduced by
L. de Branges \cite{br} in connection with inverse 
spectral problems for differential operators.
These spaces are also of a great interest from the function theory 
point of view. The Paley--Wiener space $\mathcal{P}W_a$ 
is the de Branges space corresponding to
$E(z) = \exp(-iaz)$. 

An important characteristics of the de Branges space $\he$
is its phase function, that is, an increasing $C^\infty$-function 
$\phi$ such that $E(t) \exp(i \phi(t)) \in \rl$, $t\in\rl$ 
(thus, essentially, $\phi = - \arg E$ on $\rl$).
Clearly, for $\mathcal{P}W_a$, $\phi(t) = at$.
If $\phi' \in L^\infty(\rl)$ (in which case we say that $\phi$
has sublinear growth), the space $\he$ shares certain properties
with the Paley--Wiener spaces.

A crucial property of the de Branges spaces is the existence of orthogonal bases
of reproducing kernels corresponding to real points \cite{br}.
For $\alpha\in [0, \pi)$ we consider the set of points $t_n\in\mathbb{R}$
such that
\begin{equation}
\label{basis}
\varphi(t_n)=\alpha+\pi n,\qquad n\in\mathbb{Z}.
\end{equation}
Thus, $\{t_n\}$  is the zero set of the function
$e^{i\alpha} E - e^{-i\alpha}E^*$.
It should be mentioned that the points $t_n$ may exist not for all
$n\in \mathbb{Z}$ (e.g., the sequence $\{t_n\}$ may be 
one-sided, that is, $t_n$ may exist only for $n\ge n_0$).
If the points $t_n$ are defined by (\ref{basis}),
then the system of reproducing kernels $\{K_{t_n}\}$
is an orthogonal basis for $\he$ for each $\alpha\in [0, \pi)$
except, may be, one ($\alpha$ is an exceptional value if and only if
$e^{i\alpha} E - e^{-i\alpha}E^* \in \he$).
One should think of the sequence $\{t_n\}$
as of a spectral characteristics of the space $\he$.

The completeness of a system biorthogonal to an exact system 
of reproducing kernels 
was studied in \cite{bb, fric}. In particular, it was shown in \cite{fric}
that such biorthogonal systems are always complete when $\phi'\in L^\infty(\rl)$.
The following extension of this result is obtained in \cite{bb}:
if, for some $N>0$, 
$\phi'(t) = O(|t|^N)$, $|t|\to\infty$, then either 
$e^{i\alpha} E - e^{-i\alpha}E^* \in \he$ for some $\alpha \in [0,\pi)$,
or any system biorthogonal to an exact system of reproducing kernels
is complete in $\he$.

The method of the proof of Theorem \ref{main1}
extends to the case of the de Branges spaces with sublinear growth of the phase function.

\begin{theorem}
\label{main3}
Let $\he$ be a de Branges space such that $\phi'\in L^\infty(\rl)$.
If the system of reproducing kernels 
$\{K_\lambda\}_{\lambda\in \Lambda}$ is exact in $\he$, 
then for any partition $\Lambda = \Lambda_1 \cup \Lambda_2$, 
the orthogonal complement in $\he$ to the system
\begin{equation}
\label{syst3}
\{g_\lambda\}_{\lambda\in \Lambda_1} \cup
\{K_\lambda\}_{\lambda\in \Lambda_2}
\end{equation}
is at most one-dimensional.
\end{theorem}

A crucial step in the proofs of Theorems \ref{main1}
and \ref{main3} is the use of expansions of functions
in $\pw$ or in $\he$ with respect to {\it two different} 
orthogonal bases of reproducing kernels. At first glance 
it may look like an artificial trick; however it should be noted that
the existence of two orthogonal bases of reproducing kernels
is a property which characterizes de Branges spaces 
among all Hilbert spaces of entire functions (see \cite{BMS, BMS1}).
Therefore, we believe this method to be intrinsically connected 
with the deep and complicated geometry of de Branges spaces.

As in the Paley--Wiener case, there are obstacles to the existence of the one-dimensional complement.
Here we give just a result for one-component inner functions $E^*/E$ (see, for instance, \cite{alex}) of special type.

\begin{theorem} 
\label{dop7}
Let $\he$ be a de Branges space such that $\phi'\in L^\infty(\rl)$,
$$
\sup_x\frac{|\phi(2x)|}{|\phi(x)|+1}<\infty,
$$
and
$$
\Bigl|\frac{\phi'(a)}{\phi'(b)}\Bigr|\le c\text{\ \ if \ \ }
\frac12\le \frac{\phi(a)}{\phi(b)}\le 2.
$$
Let $\Lambda\subset\mathbb R$, let the system of reproducing kernels 
$\{K_\lambda\}_{\lambda\in \Lambda}$ be exact in $\he$, 
and let the partition $\Lambda = \Lambda_1 \cup \Lambda_2$
satisfy 
$$
D^\phi_+(\Lambda_1)=\limsup_{r\to\infty} \frac{n_r(\Lambda)}{\phi(r)-\phi(-r)}>0.
$$
Then the system \eqref{syst3} is complete in $\he$. 
\end{theorem} 

Furthermore, we show that nonhereditary completeness for reproducing kernels
is possible in many de Branges spaces. Namely, we construct
such examples under some mild restrictions on the spectrum $\{t_n\}$
(including, e.g., all power growth spectra $|t_n| = |n|^\gamma$, $\gamma>0$, 
$n \in \mathbb{N}$ or $n \in \mathbb{Z}$).

\begin{theorem} 
\label{example}
Let $\{t_n\}$ be a sequence of real points such that
$t_n <t_{n+1}$ and $|t_n| \to \infty$, $n\to\infty$. Assume that
for some $N>0$, $c>0$, we have 
\begin{equation}
\label{hypot}
c|t_n|^{-N} \le t_{n+1}-t_n =o(|t_n|) , \qquad |n|\to\infty.
\end{equation}
Then there exists a de Branges space $\he$ such that 
$\{t_n\}$ is the zero set of the function $E + E^* \notin \he$
and there is an exact system of reproducing kernels $\{K_\lambda\}$ 
in $\he$ such that its biorthogonal system is complete,  but the 
original system $\{K_\lambda\}$ is nonhereditarily complete.
\end{theorem} 

We also mention here that recently Burnol \cite{BUR} studied the hereditary completeness property 
of the system $\big\{ \frac{\zeta (s)} {(s-\lambda)^k }\big\}$, where
$\lambda$ are nontrivial zeros
of the Riemann zeta function and $1\le k \le m_\lambda$, $m_\lambda$
being the multiplicity of
$\lambda$. He showed that this system is complete and minimal in
some associated space
of analytic functions, and, moreover, that this system is hereditarily
complete up to a possible
one-dimensional complement. It is not known whether this
one-dimensional defect is really possible,
but in view of our results, the presence of this complement seems to
be a sufficiently general
phenomenon.

The above counterexamples admit an operator-theoretic interpretation.
It was recently shown in \cite[Corollary 2.6]{bd}
that any exact system of reproducing kernels in a de Branges space
is unitarily equivalent to a system of eigenvectors of some rank one 
perturbation of a compact self-adjoint operator:

{\it Let $\he$ be a de Branges space such that
$e^{i\alpha} E - e^{-i\alpha}E^* \notin \he$ for any 
$\alpha \in \mathbb{R}$, and let $\{t_n\}$ be the zero set of 
$E+E^*$, $t_n \ne 0$. Put $s_n = t_n^{-1}$ and let $A$ be a compact selfadjoint 
operator with the spectrum $\{s_n\}$. Then for any 
exact system $\{K_\lambda\}_{\lambda\in \Lambda}$ 
of reproducing kernels in $\he$ 
there exists a bounded rank-one perturbation $L$ of $A$ $($i.e., 
$Lx = Ax + (x, b)\,a$$)$ such that the system $\{K_\lambda\}_{\lambda\in \Lambda}$  
is unitarily equivalent to the system of eigenvectors of 
$L$, while its biorthogonal is 
unitarily equivalent to the system of eigenvectors of
the adjoint operator $L^*$.}

Thus, we have the following corollary of Theorem \ref{example}.

\begin{corollary}
Let $\{s_n\}$ be a sequence of real numbers such that
$s_n \searrow 0$, $n\ge0$, $n\to\infty$, while 
$s_n \nearrow 0$, $n<0$, $n\to - \infty$. 
Assume also that for some $N>0$, $c>0$, we have
$$
c|s_n|^N \le |s_{n+1} - s_n| = o(|s_n|), \qquad |n| \to \infty.
$$
Let $A$ be a compact selfadjoint operator with the spectrum $\{s_n\}$
Then there exists a rank one perturbation $L$ of $A$ (with trivial kernel) such that both $L$
and $L^*$ have complete sets of eigenvectors, but
$L$ does not admit spectral synthesis.
\end{corollary}

Throughout the paper the notation $U(z)\lesssim V(z)$ (or equivalently
$V(z)\gtrsim U(z)$) means that there is a constant $C$ such that
$U(z)\leq CV(z)$ holds for all $z$ in the set in question, which may
be a Hilbert space, a set of complex numbers, or a suitable index
set. We write $U(z)\asymp V(z)$ if both $U(z)\lesssim V(z)$ and
$V(z)\lesssim U(z)$.

{\bf Acknowledgements.}
The authors are deeply grateful to Nikolai Nikolski 
who introduced them to the field of the spectral function theory 
and posed the problems studied in the paper. 
His influence on the subject is enormous.

A part of the present work was done when the authors participated 
in the research program "Complex Analysis and Spectral Problems"\,
at Centre de Recerca Matem\`atica, Barcelona. The hospitality
of CRM is greatly appreciated. 
                              

\section{Preliminaries}
\label{prelim}

Note that if $\Lambda=\Lambda_1\cup\Lambda_2$, and one of the sets $\Lambda_1$ or 
$\Lambda_2$
is finite, then the corresponding system \eqref{syst} is complete
by a simple Hilbert space argument. Therefore,
from now on we exclude the case when one of the sets 
$\Lambda_1,\Lambda_2$ is finite.

Let $h \in \pw$ be a function orthogonal to the system (\ref{syst}).
Assume that $\Lambda\cap\zl =\emptyset$
and write the expansion of the vector $h$ 
with respect to the Shannon--Kotelnikov--Whittaker 
orthonormal basis $K_n(z)=\frac{\sin \pi(z-n)}{\pi(z-n)}$,
$$
h(z) = \sum_n \overline{a_n}K_n(z) = \frac{1}{\pi}
\sum_n \overline{a_n} (-1)^n \frac{\sin \pi z}{z-n},
$$
where $a_n = \overline{h(n)}$ and 
$\|h\|^2 = \sum_n |a_n|^2<\infty$. The fact that $h$ is orthogonal to
$\big\{\frac{G(z)}{z-\lambda} \big\}_{\lambda\in \Lambda_1}$ is equivalent to
\begin{equation}
\label{1}
\bigg( \frac{G(z)}{z-\lambda}, h\bigg) = 
\frac{1}{\pi}\sum_n \frac{a_n G(n)}{n-\lambda}=0, \qquad \lambda\in \Lambda_1,
\end{equation}
while $(h, K_\lambda ) = 0$, $\lambda\in \Lambda_2$,  implies that 
\begin{equation}
\label{2}
\sum_n \frac{\overline{a_n} (-1)^n}{\lambda-n} = 0, \qquad \lambda\in \Lambda_2.
\end{equation}

Without loss of generality we may assume 
that $h$ does not vanish at integers, that is, $a_n \ne 0$, $n\in\zl$. 
Otherwise we can expand $h$ with respect to
some other basis $\{K_{n+\alpha}\}$, $\alpha\in (0,1)$.

Now let $G_2$ be an entire function of genus $1$ 
with the zero set $\Lambda_2$, and let $G_1=G/G_2$. 
The function $G_2$ is defined uniquely up to 
an exponential factor $e^{\gamma_1+\gamma z}$. 
Note that the zeros of $G$ satisfy 
the Blaschke condition in $\mathbb{C}_+$
and in $\mathbb{C}_-$. Therefore, we may choose
$\gamma$ such that $G_2^*/G_2 = B_1/B_2$ for some Blaschke products $B_1$
and $B_2$. Hence $G_1^*/G_1$ is the ratio of two Blaschke products 
as well, since $G^*/G$ is of this form for any generating
function $G$ of a complete minimal system of reproducing kernels.

We can rewrite conditions (\ref{1})--(\ref{2}) as
\begin{equation}
\label{3}
\sum_n \frac{a_n G(n)}{z-n} = \frac{G_1(z)S_1(z)}{\sin \pi z},
\end{equation}
\begin{equation}
\label{4}
\sum_n \frac{\overline{a_n} (-1)^n}{z-n} = \frac{G_2(z)S_2(z)}{\sin \pi z},
\end{equation}
where $S_1$ and $S_2$ are some entire functions.

The pairs $(S_1, S_2)$ of entire functions satisfying 
(\ref{3})--(\ref{4}) parametrize all functions orthogonal to (\ref{syst}).
We will denote the set of such pairs by $\Sigma(\Lambda_1, \Lambda_2)$.
Note that {\it the function $S_2 = h/G_2$ does not depend on the choice
of the orthogonal basis $\{k_{n+\alpha}\}$} (we will use this fact repeatedly), 
while $S_1$ will depend on the choice of the basis.

Comparing the residues at $n$
we get
\begin{equation}
S_1(n) = (-1)^n a_n G_2(n), \qquad G_2(n)S_2(n) =\overline{a_n}.
\label{residues}
\end{equation}
Put $S = S_1S_2$. Then
\begin{equation}
\label{main}
S(n) = S_1(n)S_2(n) = (-1)^n |a_n|^2.
\end{equation}

\begin{lemma}
\label{g1s1inpw}
The function $G_1S_1$ is in $\pw + z\pw$.
\end{lemma}
\begin{proof}
If $w$ is a zero of $G_1S_1$, then it follows from 
\eqref{3} and the inclusion $\{G(n)(1+|n|)^{-1}\}_n\in\ell^2$ that
\begin{equation}
\label{g-one}
 \frac{G_1(z)S_1(z)}{z-w}=\sin \pi z \sum_n\frac{a_nG(n)}{(n-w)(z-n)}\in \pw.
\end{equation}
\end{proof}

In what follows we denote by $\pw + \mathbb{C}\sin\pi z$
the class of functions of the form $f +c\sin \pi z$, where 
$f\in \pw$, $c\in \mathbb{C}$. 

\begin{lemma}
\label{growth1}
Let $h\in\pw$ be orthogonal to some system of the form (\ref{syst})
and let $(S_1, S_2) \in \Sigma(\Lambda_1, \Lambda_2)$.
Then $S \in \pw + \mathbb{C}\sin\pi z$.
\end{lemma}

\begin{proof}
Consider the function $Q\in \pw$ 
which solves the interpolation problem
$Q(n) = (-1)^n |a_n|^2$, $n\in\zl$ 
(where $a_n$ are the coefficients in the expansion $h=\sum_n 
\overline a_n K_n$) and put $\tilde S = S-Q$. Then $\tilde S$
vanishes on $\zl$ and so $\tilde S(z) = H(z) \sin \pi z$.
It remains to show that $H$ is a constant.
Note that $G_2S_2 = h\in \pw$ and, by Lemma~\ref{g1s1inpw}, 
$G_1S_1 \in \pw + z\pw$. Hence,
\begin{equation}
\label{proizv}
GS
\in \mathcal{P}W_{2\pi} +
z\mathcal{P}W_{2\pi},
\end{equation}
and, since $G\in \pw+z\pw$ and $GQ \in \ppw+z\ppw$, also
$$
G(z)\tilde S(z)  = G(z) H(z) \sin\pi z  
\in \mathcal{P}W_{2\pi} +
z\mathcal{P}W_{2\pi}.
$$
We may divide by $\sin\pi z$,  
and so
$$
G H  \in \pw+z\pw.
$$
Since $G$ is an entire function of exponential type $\pi$, we conclude that $H$
is of zero exponential type. Now if $H$ has at least one zero $z_1$,
we conclude that $\frac{H(z)G(z)}{z-z_1} \in \pw$
which contradicts the fact that $\Lambda$ is a uniqueness set for 
the Paley--Wiener space. Thus, $H$ is a constant.
\end{proof}

\begin{lemma}
\label{noexp}
Let $h\in\pw$ be orthogonal to some system of the form (\ref{syst})
and let $(S_1, S_2) \in \Sigma(\Lambda_1, \Lambda_2)$. 
Then both functions $S_1\slash S_1^*$ and $S_2\slash S_2^*$ 
are ratios of two Blaschke products.
\end{lemma}

\begin{proof}
The zero sets of $S_1$ and $S_2$ satisfy the Blaschke condition 
in $\mathbb{C}_+$ and in $\mathbb{C}_-$ since 
$G_1S_1\in \pw + z \pw$ and $h=G_2S_2 \in \pw$. Thus, it remains 
to show that  $S_1\slash S_1^*$ and $S_2\slash S_2^*$ 
have no exponential factors. By Lemma~\ref{growth1} we know that $S$ satisfies this 
property. Indeed, if $c\neq 0$ this is obvious, 
whereas if $c=0$, then the function $S$ coincides with the function 
$Q \in\pw$ which is real on $\mathbb{R}$ 
and has at least one zero in each interval $(n,n+1)$. 
So the size of the conjugate indicator diagram of the function $GS$ 
equals $4\pi$. Hence, the size of the conjugate  indicator 
diagram both for $G_1S_1$ and for $G_2S_2$ equals $2\pi$. 
Since $G_2S_2 \in \pw$, we obtain that  
$G_2S_2/(G_2^*S_2^*)$ is a ratio of two Blaschke products. 
By the construction of $G_2$, the same is true for 
$S_2$ and, hence, for $S_1$.
\end{proof}

\begin{lemma}
\label{star}
If $(S_1, S_2) \in \Sigma(\Lambda_1, \Lambda_2)$, then
also $(S_1^*, S_2^*) 
\in \Sigma(\Lambda_1, \Lambda_2)$.
\end{lemma}

\begin{proof}
By Lemma~\ref{noexp}, $S_1^*/S_1$ is of the form $B_1/B_2$ for 
some Blaschke products $B_1$ and $B_2$. 
We consider the following representation 
\begin{equation}
\frac{G_1(z)S_1(z)}{\sin \pi z}\cdot\frac{S_1^*(z)}{S_1(z)}=\sum_n\frac{a_nG(n)}{z-n}\cdot\frac{S_1^*(n)}{S_1(n)}+H(z), 
\label{starS1}
\end{equation}
where $H$ is an entire function (which holds since the residues 
at integers coincide). 
On the other hand, $G_1S_1^*\in \pw+ z \pw$,
whence $|H(z)|\lesssim 1+|z|$ and so $H$ is a polynomial of degree at most 1. 
Finally, \eqref{3} implies that $e^{-\pi |y|}|G_1(iy)S_1(iy)|\to 0$, 
$|y| \to \infty$. Since the function $S_1^*\slash S_1$ 
is reciprocal to itself at conjugate points,
we conclude that  $\min(|H(iy)|, |H(-iy)|) \to 0$, $|y|\to \infty$,
and so $H\equiv0$. 

Set $b_n=a_nS_1^*(n)/S_1(n)$. 
We can use an analogous argument to show that 
\begin{equation}
\frac{G_2(z)S_2^*(z)}{\sin \pi z}=
\sum_n\frac{\overline{a_n}(-1)^n}{z-n}\cdot\frac{S_2^*(n)}{S_2(n)}.
\label{starS2}
\end{equation}
Comparing the residues we get
$$
\overline{a_n}(-1)^nS_2^*(n)\slash S_2(n) = \overline{b_n}(-1)^n.
$$ 
Thus, the pair $(S_1^*,S_2^*)$ corresponds to the 
sequence $\{b_n\}$ in equations \eqref{3} and \eqref{4}. This means that
$(S_1^*,S_2^*)\in\Sigma(\Lambda_1,\Lambda_2)$.
\end{proof}

By Lemma~\ref{star},
if $(S_1, S_2)\in\Sigma (\Lambda_1, \Lambda_2)$, then 
$(S_1+S_1^*, S_2+S_2^* )\in\Sigma (\Lambda_1, \Lambda_2)$
and $(iS_1 - iS_1^*,-iS_2 +i S_2^*) \in\Sigma (\Lambda_1, \Lambda_2)$.
Thus, in what follows we may assume that the functions $S_1$ and 
$S_2$ are real on $\RR$. 
In this case we have an immediate corollary from (\ref{main}).

\begin{corollary}
\label{zeros1}
If $S_1$ and $S_2$ are real on $\RR$, then each open
interval $(n, n+1)$, $n\in \mathbb{Z}$, contains exactly one zero of $S$, 
and $S$ has no other zeros.
\end{corollary}

\begin{proof}
Since $S$ is real on $\RR$ and changes the sign at $n \in \mathbb{Z}$, 
it has at least one zero in every interval $(n, n+1)$.
Choosing a zero in each interval we construct
the (principal value) canonical product $S_0$. 
Then $S=S_0 H$ for some entire function $H$
of zero exponential type which is real on $\RR$. 
Clearly, $|S_0(iy)| \gtrsim |y|^{-1}e^{\pi |y|}$, $|y|\to\infty$. 
By Lemma~\ref{growth1} we have $S\in \pw + \mathbb{C}\sin\pi z$.
Hence, $|H(iy)| \lesssim |y|$, $|y| \to\infty$, 
which implies that $H$ is a polynomial of degree at most 1.
Since the signs of $S(n)$ interchange, $S$ cannot have two zeros
in any of the intervals $(n, n+1)$. Thus, $H$
is a constant.
\end{proof}

\section{Proofs of Theorems \ref{main1} and \ref{main2}}

We are now ready to prove the main positive 
results on hereditary completeness for exponential systems.

\subsection{Completeness up to a one-dimensional defect}
\medskip
\noindent
{\it Proof of Theorem \ref{main1}.}
Without loss of generality assume that $\Lambda \cap \zl =\emptyset$.
Let $f = \sum_{n\in\zl} \overline{a_n} K_n$ 
and $h=\sum_{n\in\zl} \overline{b_n} K_n$ be two linearly independent 
vectors orthogonal to (\ref{syst}), and let 
$(S_1, S_2)$ and $(T_1, T_2)$  be the corresponding pairs
of entire functions from $\Sigma (\Lambda_1, \Lambda_2)$. 
Since, by Lemma~\ref{star}, 
the pairs
$(S_1+S_1^*, S_2+S_2^*)$,
$(iS_1-iS_1^*, -iS_2 +i S_2^*)$,
$(T_1+T_1^*, T_2+T_2^*)$,  and 
$(iT_1 - iT_1^*, -iT_2  +i T_2^*)$
also belong to $\Sigma (\Lambda_1, \Lambda_2)$, we may assume from the 
very beginning that the pairs $(S_1, S_2)$ and $(T_1, T_2)$
are linearly independent and 
the functions $S_1$, $S_2$, $T_1$, and $T_2$  
are real on $\mathbb{R}$.

Using equations \eqref{residues} for $S$ and $T$ we get
$$
S_1(n) T_2(n) \overline{G_2(n)} = 
T_1(n) S_2(n) \overline{G_2(n)} = 
(-1)^n G_2(n) a_nb_n, 
$$
and hence,
$$
S_1(n)T_2(n) = 
S_2(n)T_1(n) = \beta_n a_n b_n, 
$$
with $|\beta_n| = 1$.

Denote by $Q$ the function
in $PW_\pi$ which solves the interpolation problem
$Q(n) = \beta_n  a_nb_n$. Then 
$$
T_1(z) S_2(z) = Q(z) +a(z) \sin \pi z, \qquad 
S_1(z) T_2(z) = Q(z) +b(z) \sin \pi z,
$$  
for some entire functions $a$ and $b$.
We show now that $a$ and $b$ are constants. 

Note that the functions $S=S_1S_2$ and $T=T_1T_2$ 
are in $\pw + \mathbb{C} \sin\pi z$ by Lemma~\ref{growth1}.
Furthermore,   
the pair $(S_1+T_1, S_2 +T_2)$ 
corresponds to the vector $f+h$
while the pair 
$(S_1+iT_1,S_2-iT_2)$
corresponds to the vector $f+ih$.
Applying again Lemma~\ref{growth1} we obtain that 
$U=(S_1+T_1)(S_2+T_2)$ and $V = (S_1+iT_1)(S_2-iT_2)$
are in $\pw+ \mathbb{C} \sin \pi z$. Hence the functions
$$
S_1 T_2 + S_2T_1  = U-S-T, \qquad i(S_2 T_1 - S_1T_2 ) = V-S-T
$$
belong to $\pw + \mathbb{C} \sin \pi z$.
Thus, $S_1T_2$, $S_2T_1 \in \pw + \mathbb{C} \sin\pi z$, and we conclude that $a$
and $b$ are constants. 

Assume that $a \ne 0$.
Let us denote by $s_m$ the zero of $S_2$ in the interval $[m-1/2, m+1/2]$
for those $m$ for which such a zero exists. Then
$$
Q(s_m) + a  (-1)^m \sin \pi (s_m-m) = 0,
$$
whence 
$$
\sum |s_m-m|^2 \asymp \sum \sin^2 \pi (s_m-m) \asymp
\sum |Q(s_m)|^2<\infty.
$$
On the other hand, the zeros of $S_2$ do not depend on the choice 
of the basis, they are the zeros of $h/G_2$. Expanding with respect to
another basis (say, $\{n+\delta\}$ with small $\delta$) 
we conclude that $\sum |s_m-m-\delta|^2 <\infty$.
This is obviously wrong.

Thus, we have proved that $a =b=0$, and so $S_1 T_2= T_1 S_2 = Q$.
Since $S_1$ has no common zeros with $S_2$ (we choose the basis so that all $a_n$
are nonzero), and the same is true for $T_1,T_2$, we conclude that the zero sets of $S_2$ and  $T_2$ coincide, 
and, thus, $g =ch$ for some constant $c$, a contradiction.
\qed

\subsection{Proof of Theorem \ref{main2}}

The following proposition plays the key role in the proof of 
Theorem \ref{main2}. In Section~\ref{secdeb} we prove 
a slightly stronger result which applies to general
de~Branges spaces (see Proposition~\ref{closeness1}).
We prefer, however, to include an elementary proof
to make the exposition concerning exponential systems
more self-contained.

\begin{proposition}
\label{closeness}
Let $S\in \pw +\mathbb{C}\sin \pi z$  be a real entire function
with real zeros $\mathcal{Z}_S$ interlacing with 
$\mathbb{Z}$. If $\sum_{n\in\Z}|S(n)| <\infty$, then for every $\delta>0$
we have
$$
L_\delta:=\lim_{N\to\infty} \frac{1}{N} \,{\rm card}\,  \big\{
|k| \le N: {\rm dist}\,(\mathcal{Z}_S\cap [k, k+1], \Z)>\delta \big\} =0.
$$
\end{proposition}

\begin{proof}
Let $S(n) = (-1)^n c_n$. Without loss of generality we may assume that
$c_n>0$ and $\sum_{n\in\Z} c_n = 1$. Then $S(z)/\sin\pi z$ 
is a Herglotz function in $\CC_+$ and 
$$
\frac{S(z)}{\sin\pi z} = b +\sum_{n\in\Z} \frac{c_n}{z-n}
$$
for some $b\in\RR$. Set $s(x) = \sum_{n\in\Z} \dfrac{c_n}{x-n}$.
\medskip
\\
{\bf Case 1.} If $b\ne 0$, then 
$$
\lim_{x\in \mathcal{Z}_S, |x|\to\infty} {\rm dist}\, (x, \Z) = 0. 
$$ 
This follows from the fact that for any $\delta>0$ we have 
$s(x) \to 0$ as $|x| \to\infty$ and ${\rm dist}\, (x, \Z) \ge \delta$.
\medskip
\\
{\bf Case 2.} Suppose  that $b= 0$. Fix two positive 
numbers $\delta<1/4$ and $\eta <\delta^3$ 
and choose $M$ so that $\sum_{|n| \le M} c_n > 1-\eta$. 

Now let the integer $N$ be so large that $\delta N>M$. Put 
$$
E_N = \Big\{x\in\RR: \, \Big|\sum_{n\in\Z} \frac{c_n}{x-n}\Big|\ge \frac{1}{N} 
\Big\}.
$$
By Boole's lemma, $|E_N|= 2N$ (by $|E|$ we denote the Lebesgue measure of 
the set $E$).

Next, set 
$$
F_N = \Big\{x\in\RR: \, \Big| \sum_{|n|>M} 
\frac{c_n}{x-n}\Big|\ge \frac{\delta}{2N} 
\Big\}.
$$
Then
$$
|F_N| \le \frac{4N\eta}{\delta}.
$$

Let $J_N = [-N-\delta N - M, N+\delta N+M]$. 
Since 
$$
\Big|\sum_{|n|\le M} \frac{c_n}{x-n}\Big| \le \frac{1}{(1+\delta)N}, \qquad x\notin J_N,
$$
we have, for $x\in E_N\setminus J_N$,
$$
\Big|\sum_{|n|> M} \frac{c_n}{x-n}\Big| \ge \frac{1}{N} - \frac{1}{(1+\delta)N} = 
\frac{\delta}{(1+\delta)N},
$$
and so $x\in F_N$. We conclude that $E_N \setminus J_N\subset F_N $. 

Consider the family $\mathcal{I}_N$ of the intervals of the form 
$I_k = [k, k+1] \subset J_N$ with $|k|\ge M+\delta N$ 
satisfying the following two properties: 
\begin{equation}
\label{prop1}
(I_k^* \cap E_N)\setminus F_N \ne \emptyset, \qquad I_k^* = [k+\delta, k+1-\delta];
\end{equation}
\begin{equation}
\label{prop2}
|I_k \cap F_N| < \delta.
\end{equation}

We will show that, for sufficiently large $N$, we have 
\begin{equation}
\label{prop3}
{\rm card}\, \mathcal{I}_N \ge (1 - A_1 \delta)|J_N|,
\end{equation}
where $A_1$ is some absolute (numeric) constant.
In what follows, symbols $A_1$, $A_2$, etc. will denote different
absolute constants.

If $(I_k^* \cap E_N)\setminus F_N = \emptyset$
(i.e., the interval $I_k^*$  does not satisfy (\ref{prop1})), then
$I_k^* \subset (J_N \setminus E_N) \cup F_N$ and
$$
\begin{aligned}
|(J_N \setminus E_N) \cup F_N| & \le |J_N| - |J_N\cap E_N|+|F_N| \\
& = 
|J_N|-|E_N|+|E_N\setminus J_N|+|F_N|  
\\ 
& \le 2N +2\delta N +2M -2N + \frac{8 N \eta}{\delta} \le A_2 \delta N.
\end{aligned}
$$ 
Hence, for the number $N_1$ of those intervals $I_k^*$ which do not satisfy 
(\ref{prop1}), we have the estimate
$$
N_1(1-2\delta) \le A_3\delta N.
$$
On the other hand,
for the number $N_2$ of those intervals $I_k$ which do not satisfy (\ref{prop2}),
we get $N_2\delta \le \frac{4N\eta}{\delta}$,
and so $N_2 \le \frac{4N \eta}{\delta^2} \le A_4 \delta N$,
since $\eta<\delta^3$. Thus, for sufficiently large $N$,
$$
{\rm card}\, \mathcal{I}_N \ge 2N - N_1 - N_2 \ge 2N - A_5\delta N.
$$
The latter inequality implies (\ref{prop3}).

Now, if $I_k\in \mathcal{I}_N$, then there exists a point 
$y\in (I_k^* \cap E_N)\setminus F_N$ and so we have
$$
\Big|\sum_{|n|\le M} \frac{c_n}{y-n}\Big| \ge \frac{1}{N} - \frac{\delta}{2N}> 
\frac{1}{2N}.
$$
For any $x\in I_k^*$ using the fact that $|k|\ge M+\delta N$ we get
\begin{equation}
\label{prop4}
\Big|\sum_{|n|\le M} \frac{c_n}{x-n} - \sum_{|n|\le M} \frac{c_n}{y-n}\Big|
\le \sum_{|n|\le M} \frac{c_n|x-y|}{|(x-n)(y-n)|} \le \frac{1}{\delta^2 N^2}
\le \frac{1}{4N} 
\end{equation}
for sufficiently large $N$, and hence,
$$
\Big|\sum_{|n|\le M} \frac{c_n}{x-n}\Big| \ge \frac{1}{4N}, \qquad x\in I_k^*.
$$

Suppose that for some $w\in I_k^*$ we have 
$s(w)=\sum_{n\in\Z} \frac{c_n}{w-n} = 0$. Then
$$
\Big|\sum_{|n|> M} \frac{c_n}{w-n}\Big| \ge \frac{1}{4N}>\frac{\delta}{N}.
$$
So $w\in F_N$ and, moreover, since the function under the modulus sign
is monotone on $I_k$  we obtain that either $[k, w] \subset F_N$
or $[w, k+1] \subset F_N$, which is impossible due to (\ref{prop2}).

Thus, the zeros of $s$ (and hence of $S$) on $I_k\in \mathcal{I}_N$ 
are in $I_k\setminus I_k^*$. It follows from (\ref{prop3}) that
$$
L_\delta = \limsup_{N\to\infty} \frac{1}{N} {\rm card}\,  \big\{
|k| \le N: {\rm dist}\,(\mathcal{Z}_S\cap [k, k+1], \Z)>\delta
\big\} \le A \delta
$$
for some absolute constant $A$. 
Since $L_\delta$ is a non-increasing nonnegative function 
of $\delta$ on $(0, 1/4)$, it follows that $L_\delta \equiv 0$.
\end{proof}
\medskip
\noindent
{\it Proof of Theorem \ref{main2}.}
Assume that there is a nontrivial function $h$ orthogonal to
the system (\ref{syst}) such that
$D_+(\Lambda_1)>0$.
Denote by $\mathcal{Z}_1$ and
$\mathcal{Z}_2$
the zero sets of $S_1$ and $S_2$, respectively.

Since $G_1S_1 \in \pw + z \pw$, by the Levinson theorem (see, for instance, \cite[Section IIIH3]{Koo}) we have
$$
D(\Lambda_1 \cup \mathcal{Z}_1) = 
 \lim_{r\to\infty} \frac{n_r( \Lambda_1 \cup \mathcal{Z}_1) }
{2r} \le \pi,
$$
and so 
$$
D_-(\mathcal{Z}_1) = \liminf_{r\to\infty} \frac{n_r(\mathcal{Z}_1)}
{2r} <\pi.
$$
Since $S$ is of exponential type $\pi$, we have
$D_+(\mathcal{Z}_2)>0$.

The function $S_2= h/G_2$ 
does not depend on the choice of the basis, and replacing if necessary
the basis $\{K_n\}$ by the basis $\{K_{n+\alpha}\}$ 
we may find $\alpha$ such that for a subsequence $\tilde{\mathcal{Z}}_2$
of $\mathcal{Z}_2$  with positive upper density we have 
${\rm dist}\,( \tilde{\mathcal{Z}}_2, \Z+\alpha)\ge 1/4$.
Without loss of generality assume that this holds for $\alpha = 0$.
Construct the function $S_1$ corresponding to this basis
by formula (\ref{3}). 
Then for $S=S_1S_2$ we have $\sum_{n\in \Z} |S(n)|<\infty$.
Note that by Corollary \ref{zeros1}
the zeros of $S$ interlace with $\Z$.
By Proposition~\ref{closeness} all zeros of $S$ except 
the set of zero density
are close to $\Z$, and we come to a contradiction.
\qed


\section{An example of a nonhereditarily complete exponential system}

In this section we prove Theorem \ref{maincount}. 
As before, we pass to the equivalent problem in the Paley--Wiener space
and construct a nonhereditarily complete system
of reproducing kernels $\{k_{\lambda}\}_{\lambda \in \Lambda}$ in $\pw$. 

We deduce Theorem \ref{maincount} from the following statement.

\begin{proposition}
\label{technical}
There exist a sequence $\{a_n\} \in \ell^1(\mathbb{Z})$ 
such that $a_n > 0$, and an infinite 
sequence $\{n_k\}_{k=1}^\infty \subset \na$, 
$n_{k+1}>2n_k$, $k\ge 1$, such that the functions
$$
h(z) = \sin \pi z \sum_{n\in \zl} \frac{a_n}{z-n}, \qquad 
S(z) = \sin \pi z \sum_{n\in \zl} \frac{a_n^2}{z-n}
$$
vanish at the points $s_k = n_k+1/2$, $k\in\na$, and 
$a_{n_k} = \alpha_k k^{-2}$ with $\alpha_k \in (1,3)$, $k\in\na$.
\end{proposition}

Proposition \ref{technical} is proved using standard 
fixed point arguments of nonlinear analysis. 
We postpone its (rather technical) proof and show first 
how Theorem \ref{maincount} follows from Proposition \ref{technical}.
\medskip
\\
{\it Proof of Theorem \ref{maincount}.}
We have seen in Section \ref{prelim}
that if $\{k_{\lambda}\}_{\lambda \in \Lambda}$ 
is a complete minimal system in $\pw$
with a generating function $G$
and $\Lambda = \Lambda_1\cup\Lambda_2$, then 
we may construct entire functions 
$G_1$ and $G_2$ with zero sets $\Lambda_1$ and $\Lambda_2$
respectively such that each of the functions 
$G_1^*/G_1$ and $G_2^*/G_2$ is a ratio 
of two Blaschke products and $G=G_1G_2$.
Once such functions $G_1$ and $G_2$ are chosen, we have seen
that the system (\ref{syst}) is not complete in $\pw$
if and only if there exists a nonzero sequence $\{a_n\} \in \ell^2$
and entire functions $S_1$ and $S_2$ satisfying the equations 
(\ref{3})--(\ref{4}).

We first choose $S_1$, $S_2$ and $G_2$, and finally construct $G_1$ 
as a perturbation of $S_2$. Let $\{a_n\} \in \ell^1$ and 
$\{n_k\}\subset \na$ be the sequences from Proposition \ref{technical}. 
As in Proposition \ref{technical} put
\begin{align}
h(z) &= \sin \pi z \sum_{n\in \zl} \frac{a_n}{z-n}, \label{dx1}\\ 
S(z) &= \sin \pi z \sum_{n\in \zl} \frac{a_n^2}{z-n}.\notag
\end{align}
Note that for the functions $h$ and $S$ we have
\begin{equation}
\label{5c}
\frac{|h(iy)|}{e^{\pi |y|}}\asymp 
\frac{1}{|y|}, 
\qquad  
\frac{|S(iy)|}{e^{\pi |y|}}\asymp  \frac{1}{|y|}, 
\qquad |y|\ge 1.
\end{equation}

Denote by $S_2$ the genus zero canonical product with the zeros
$n_k+\frac{1}{2}$. Then we may represent $h$ and $S$ as
$$
h=G_2S_2, \qquad S=S_1S_2
$$
for some entire functions $S_1$ and $S_2$. Since $a_n>0$ for any $n\in \mathbb Z$,
the function $h(z)/\sin \pi z$ is a Herglotz function
and so all the zeros of $h$ (and thus of $G_2$) are simple and real.

We need to show that there exists an entire function $G_1$ 
with simple real zeros (different from the zeros of $G_2$) such that
$G=G_1G_2$ is the generating function of some complete and minimal system 
of reproducing kernels, and
\begin{equation}
\label{6c}
\frac{G_1(z)S_1(z)}{\sin \pi z} =  \sum_{n\in \zl} \frac{a_n(-1)^n G(n)}{z-n}
\end{equation}
(this equation is actually equation \eqref{3} for the sequence 
$\{(-1)^na_n\}_n$). Note that \eqref{dx1} gives us 
\eqref{4} for the same sequence.

Put
$$
G_1(z) = \prod\limits_{k=1}^\infty \bigg(1- \frac{z}{n_k+\frac{1}{2} - k^2}\bigg).
$$
Then an easy estimate of the infinite products give us
$$
\frac{|G_1(n)|}{|S_2(n)|} \asymp \frac{|n+k^2 -n_k -\frac{1}{2}|}
{|n-n_k -\frac{1}{2}|}, \qquad
\frac{n_{k-1}+n_k}{2} \le n \le
\frac{n_{k}+n_{k+1}}{2}, 
$$
whence, in particular,
$$
\frac{|G_1(n_k)|}{|S_2(n_k)|} \asymp k^2 \qquad
\text{and} \qquad 
\frac{|G_1(n)|}{|S_2(n)|} \lesssim |n|^{1/2}, \quad n \ne 0.
$$
Since $G = hG_1/S_2$, we have 
$$
|G(n_k)| = \frac{|h(n_k) G_1(n_k)|}{|S_2(n_k)|}
\asymp a_{n_k} k^2 \asymp 1
$$
(recall that in Proposition \ref{technical}, 
$|h(n_k)| = a_{n_k} = \alpha_k k^{-2}$, $\alpha_k \in (1,3)$), 
and 
$$
|G(n)|\lesssim |n|^{1/2} |h(n)|, \qquad n\ne 0.
$$
Hence, $\{G(n)\}_{n\in\zl} \notin \ell^2$, and thus $G\notin \pw$.
However, 
$\Big\{\frac{G(n)}{|n|+1}\Big\}_{n\in\zl} \in \ell^2$, and, using the fact that 
\begin{equation}
\label{7c}
\frac{|G(iy)|}{e^{\pi |y|}} 
\asymp\frac{1}{|y|}, \qquad |y|\ge 1,
\end{equation}
we conclude that $\frac{G(z)}{z-\lambda} \in \pw$
for any zero $\lambda$ of $G$.
Now let us turn to the formula (\ref{6c}). Comparing the residues
at $n\in \zl$  we have $S_1(n)S_2(n) =(-1)^n a_n^2$
and $S_2(n)G_2(n) = a_n(-1)^n$ whence $G_1(n) S_1(n) = a_nG(n)$.
Therefore, the residues in the left and the right-hand sides of (\ref{6c})
coincide. Hence,
$$
\frac{G_1(z)S_1(z)}{\sin \pi z} =  \sum_{n\in \zl} 
\frac{a_n(-1)^n G(n)}{z-n} +H(z)
$$
for some entire function $H$. By the standard growth arguments
$H$ is of zero exponential type.
Note also that $G_1S_1 = S G_1/S_2$ whence, by (\ref{5c}) and the fact that 
$|G_1(iy)| \asymp |S_2(iy)|$, $|y|\to\infty$, we get 
$$
\frac{|G_1(iy) S_1(iy)|}{e^{\pi |y|}}\asymp  \frac{1}{|y|}, 
\qquad |y| \to \infty.
$$
Thus, $H(iy) \to 0$, $|y| \to \infty$, whence 
$H\equiv 0$ and \eqref{6c} is proved. 

It remains to show that $G$ is the generating function of 
a complete and minimal system of reproducing kernels. 
We have already seen
that $G\notin \pw$ but $G\in \pw +z\pw$. Assume now that the zero set 
$\Lambda$ of $G$ is not a uniqueness set for $\pw$. Then there exists
a function $T$ of zero exponential type such that $TG\in \pw$.
Hence, $e^{i\pi z} TG \in H^2(\co_+)$, $e^{-i\pi z} T^*G^* \in H^2(\co_-)$, 
and
it follows from (\ref{7c}) that
$$
\frac{|T(iy)|}{|y|} \asymp \frac{|T(iy)G(iy)|}{e^{\pi |y|}}  \lesssim |y|^{-1/2}, 
\qquad |y|\ge 1.
$$              
Thus $T$ is a constant function whence $T\equiv 0$.
\qed
\bigskip
\\
{\it Proof of Proposition \ref{technical}.}
We will construct the sequence $a_n$ as follows:
let $a_0$ be an arbitrary positive number, 
$a_n = n^{-2}$ for $n \ne 0$ and for $n\ne n_k, n_{k}+1, n_{k}+2$, while
$$
a_{n_{k}} = 2 r_{2k-1} k^{-2}, \qquad 
a_{n_{k}+1} = r_{2k} k^{-2}, \qquad
a_{n_{k}+2} = 3k^{-2} 
$$
for some free parameters $r_{2k-1}$ and $r_{2k}$.
Here $n_k$ is some very sparse
subsequence of positive integers.
The sparseness condition is to be specified
later. Thus the coefficients $a_{n_k}$  have a much slower decay than all
other coefficients.

Using basic tools of nonlinear analysis 
we will show that it is possible to find 
parameters $r_{2k-1}, r_{2k} \in (1/2, 3/2)$, $k\in \na$, 
such that 
\begin{equation}
\label{req}
h(s_k) = S(s_k) = 0, \qquad k\in \na,\quad s_k=n_k+\frac{1}{2}.
\end{equation}
Denote by $\nar$ the set $\bigcup_k\{n_k\}\cup \{n_k+1\}$.
Clearly, (\ref{req}) is equivalent to the system of equations
\begin{equation}
\label{eq1}
\sum_{l=1}^\infty \bigg(\frac{2r_{2l-1}}{l^2 (s_k - n_l)} +
\frac{r_{2l}}{l^2 (s_k - n_l-1)}\bigg) = 
- \sum_{n \notin \nar} \frac{a_n}{s_k-n}
\end{equation}
and 
\begin{equation}
\label{eq2}
\sum_{l=1}^\infty \bigg(\frac{4 r^2_{2l-1}}{l^4 (s_k - n_l)} +
\frac{r^2_{2l}}{l^4 (s_k - n_l-1)}\bigg) = 
- \sum_{n \notin \nar} \frac{a_n^2}{s_k-n}.
\end{equation}
Multiply the equation (\ref{eq1}) by $k^2/2$
and (\ref{eq2}) by $k^4/2$. Using the fact that $s_k = n_k + 1/2$
and that $a_{n_{k}+2} = 3 k^{-2}$,
we may single out the diagonal part which will form the main contribution
to the equations:
\begin{equation}
\label{eq1n}
2r_{2k-1} - r_{2k} +\sum_{l\ne k}
\bigg(\frac{k^2 r_{2l-1}}{l^2 (s_k - n_l)} +
\frac{k^2 r_{2l}}{2l^2 (s_k - n_l-1)}\bigg) =  1-
\sum_{n \notin \nar, n\ne n_{k}+2} \frac{k^2a_n}{2(s_k-n)},
\end{equation}
\begin{equation}
\label{eq2n}
4r^2_{2k-1} - r^2_{2k} +\sum_{l\ne k}
\bigg(\frac{2k^4 r^2_{2l-1}}{l^4 (s_k - n_l)} +
\frac{k^4 r^2_{2l}}{2l^4 (s_k - n_l-1)}\bigg) =  3-
\sum_{n \notin \nar, n\ne n_{k}+2} \frac{k^4 a_n}{2(s_k-n)}.
\end{equation}
\medskip
\\
{\bf Diagonal part of the map.} 
Denote by $r$ the vector $(r_j)_{j=1}^\infty$
and consider the nonlinear mapping $D: \ell^\infty \to \ell^\infty$,
$$
\begin{aligned}
(Dr)_{2k-1} & = 2r_{2k-1} - r_{2k}, \\
(Dr)_{2k} & =  4r^2_{2k-1} - r^2_{2k}. 
\end{aligned} 
$$
Thus, $D$ is a block-diagonal mapping and the solution
of the equation $D(r) = y$ is given by
\begin{equation}
\label{d-1}
(D^{-1}y)_{2k-1} =  r_{2k-1}  = \frac{y_{2k} +y^2_{2k-1}}{4y_{2k-1}}, \quad
(D^{-1}y)_{2k} = r_{2k}  =  \frac{y_{2k}  - y^2_{2k-1}}{2y_{2k-1}},
\quad y_{2k-1}\not=0. 
\end{equation}
If we set $r^\circ_j \equiv 1$, then $D(r^\circ) = y^\circ$
with $y^\circ_{2k-1} = 1$, $y^\circ_{2k} = 3$.
Next, for any $y\in \ell^\infty$ such that $\|y-y^\circ\|_\infty <1/2$
there exists a unique solution $r\in \ell^\infty$ 
of the equation $D(r) = y$.

Moreover, it is easy to see from the
form (\ref{d-1}) of the block-diagonal mapping $D^{-1}$ 
that there exists an absolute constant $A_0>0$ such that
\begin{equation}
\label{dm1}
\|D^{-1} (y) - D^{-1}(z)\|_\infty\le A_0\|y-z\|_\infty
\end{equation}
for all $y, z\in \ell^\infty $ such that $\|y-y^\circ\|_\infty <1/2$, 
$\|z-y^\circ\|_\infty <1/2$.

We need one more estimate for the mapping $D^{-1}$. Put
$F(u) = \frac{u_2 + u_1^2}{4u_1}$, $u=(u_1, u_2)$. 
Then an elementary estimate gives us
$$
|(F(u+\Delta u) - F(u))- (F(v+\Delta v) - F(v))| \le
A_1 \|\Delta u - \Delta v\|_\infty + A_2 \|\Delta v\|_\infty\|u-v\|_\infty
$$
for some absolute constants $A_1$ and $A_2$
whenever $u_1, v_1 \in (1/2, 3/2)$, $u_2, v_2 \in (2, 4)$ 
and $\|\Delta u\|_\infty\le 1/10$, $\|\Delta v\|_\infty\le 1/10$.
An analogous estimate holds for
$F(u) = \frac{u_2 - u_1^2}{2 u_1}$.
Hence, taking into account formula (\ref{d-1}) for $D^{-1}$, we conclude 
that
\begin{equation}
\label{dm2}
\begin{aligned}
\| (D^{-1} (y+\Delta y)  - D^{-1}(y)) &- (D^{-1}(z+\Delta z) - D^{-1}(z))\|_\infty \\
& \le A_1 \|\Delta y - \Delta z\|_\infty + A_2 \|\Delta z\|_\infty \|y-z\|_\infty
\end{aligned}
\end{equation}
for all $y,z\in \ell^\infty$ such that $\|y-y^\circ\|_\infty <1/2$,
$\|z-z^\circ\|_\infty <1/2$, and 
$\|\Delta y\|_\infty <1/10$, $\|\Delta z\|_\infty <1/10$.

Finally we will need the following obvious estimate: there exists
an absolute constant $A_3>0$ such that 
\begin{equation}
\label{dm3}
\|D(r) - D(s)\|_\infty \le A_3 \|r-s\|_\infty, \qquad 
\|r\|_\infty\le 10, 
\ \|s\|_\infty \le 10.
\end{equation}
\medskip
\\
{\bf Sparseness conditions on $\{n_k\}$.} 
Now we impose the {\it first sparseness condition} on the sequence $n_k$:
\begin{equation}
\label{sp1}
\sum_{n \notin \nar, n\ne n_{k}+2} \Big|\frac{k^2a_n}{2(s_k-n)}\Big| 
+
\sum_{n \notin \nar, n\ne n_{k}+2} \Big|\frac{k^4 a_n}{2(s_k-n)}\Big| 
<\frac{1}{200 (A_0+1)}, \qquad k\in \na
\end{equation}
(where $A_0$ is the constant from (\ref{dm1})). 
Since $a_n = n^{-2}$, $n\notin \nar \cup\{n_l+2\}_{l=1}^\infty$,
we have $|a_n|\asymp n_k^{-2}$, 
$n\in [n_k/2, 2n_k]$, $n\ne n_k, n_k+1, n_k+2$, and so the terms
$$
\bigg|\frac{k^2a_n}{2(s_k-n)}\bigg|, \qquad
\bigg|\frac{k^4 a_n}{2(s_k-n)}\bigg|  
$$
may be made arbitrarily small when $n_k$ grows sufficiently fast.
E.g., we may take $n_k = M2^k$ with a sufficiently large constant 
$M$.

Let us consider the vector $y^* \in \ell^\infty$ defined by
$$
y^*_{2k-1} =  1 -
\sum_{n \notin \nar, n\ne n_{k}+2} \frac{k^2a_n}{2(s_k-n)},
\qquad 
y^*_{2k} = 3-
\sum_{n \notin \nar, n\ne n_{k}+2} \frac{k^4 a_n}{2(s_k-n)}.
$$
By (\ref{sp1}), 
$\|y^* - y^\circ\|_\infty < (200(A_0+1))^{-1}$. Hence, 
there exists $r^*$ such that $D(r^*) = y^*$ and, by (\ref{dm1}),
$\|r^* - r^\circ\|_\infty <A_0 \cdot(200(A_0+1))^{-1} <1/200$.

Next we define the mapping $W$ corresponding to the nondiagonal
part of the equations (\ref{eq1n})--(\ref{eq2n}):
$$
\begin{aligned}
(Wr)_{2k-1} & =  \sum_{l\ne k}
\bigg(\frac{k^2 r_{2l-1}}{l^2 (s_k - n_l)} +
\frac{k^2 r_{2l}}{2l^2 (s_k - n_l-1)}\bigg), \\
(Wr)_{2k} & =  \sum_{l\ne k}
\bigg(\frac{2k^4 r^2_{2l-1}}{l^4 (s_k - n_l)} +
\frac{k^4 r^2_{2l}}{2l^4 (s_k - n_l-1)}\bigg).
\end{aligned}
$$
Choosing the sequence $n_k$ sufficiently sparse
(again $n_k = M2^k$ will do the job)
we may achieve our {\it second and third sparseness conditions}:
\begin{equation}
\label{sp2}
\|W(r)\|_\infty \le \frac{1}{200 (1 + A_0 + A_2 A_3)},  \qquad \|r\|_\infty \le 10, 
\end{equation}
and \begin{equation}
\label{sp3}
\|W(r) - W(s)\|_\infty \le \frac{\|r-s\|_\infty}{200A_1},  \qquad \|r\|_\infty\le 10, 
\ \|s\|_\infty \le 10,
\end{equation}
where $A_1$, $A_2$ and $A_3$ are constants from (\ref{dm2})--(\ref{dm3}).
\medskip
\\
{\bf Application of the Fixed Point Theorem.} 
Equations (\ref{eq1n})--(\ref{eq2n}) are equivalent to
$$
D(r) +W(r) = y^*.
$$
Consider the mapping 
$$
T(r) = r^* + r - D^{-1}(D(r) +W(r)).
$$
We show that $T$ is a contractive mapping on the ball  
$B = \{\|r-r^\circ\|_\infty \le 1/100$\}. Then there exists
$r\in B$ such that $T(r) = r$ which is equivalent to 
$D^{-1}(D(r) +W(r)) =r^*$, whence $D(r) +W(r) = D(r^*) = y^*$.
\medskip

{\it 1. $T$ is well-defined on $B$.} Clearly, 
we have $\|D(r) - D(r^\circ)\|_\infty <1/4$  
and $\|W(r)\|_\infty <1/200$
when $\|r-r^\circ\| \le 1/100$.
Thus, $\|D(r)+W(r) - y^\circ\|_\infty < 1/2$
and so $D^{-1}(D(r) +W(r))$ is well-defined.

\medskip
{\it 2. $T(B) \subset B$.} 
We have already seen that $\|r^* - r^\circ\|_\infty  <1/200$. Then
$$
\begin{aligned}
\|T(r) - r^\circ\|_\infty  & \le \|r^* - r^\circ\|_\infty +
\|D^{-1}(D(r)) - D^{-1}(D(r) +W(r))\|_\infty \\
& <\frac{1}{200} + A_0 \|W(r)\|_\infty <\frac{1}{100}
\end{aligned}
$$
by (\ref{dm1}) and (\ref{sp2}).

\medskip
{\it 3. $T$ is a contraction on $B$.} Let $r, s \in B$.
Then
$$
T(r) -T(s) = \big(D^{-1}(D(s) +W(s)) -  D^{-1}(D(s)\big) 
- \big(D^{-1}(D(r) +W(r))- D^{-1}(r)\big).
$$
By (\ref{dm2}) applied to $y=D(s)$, $\Delta y = W(s)$,
and $z=D(r)$, $\Delta z = W(r)$, we have 
$$
\begin{aligned}
\|T(r) -T(s)\|_\infty & \le A_1\|W(s) - W(r)\|_\infty 
+A_2 \|W(r)\|_\infty\|D(s)-D(r)\|_\infty \\
& \le  A_1\|W(s) - W(r)\|_\infty  + A_2 A_3 \|W(r)\|_\infty\|r-s\|_\infty \le
\frac{\|r-s\|_\infty}{100}.
\end{aligned}
$$
We used estimate (\ref{dm3}) in the second inequality  
and (\ref{sp2}) and (\ref{sp3}) in the last  one.
Thus, $T$ is a contractive mapping from $B$ to $B$ 
and we conclude that $T$ has a fixed point.
\qed

\section{Extensions to the de Branges spaces}
\label{secdeb}

\subsection{Preliminary remarks}\label{51}
We start with a general construction of 
functions biorthogonal to a system of reproducing kernels.
Let $\he$ be a de Branges space, and let $\phi$ 
be the corresponding phase function. 
As usual, we write $E=A-iB$. To avoid inessential difficulties we will always
assume that 
$A\notin \he$. 
The reproducing kernel of ${\mathcal H} (E)$ can be written as 
$$
K_w(z)= 
\frac{\overline{A(w)} B(z) -\overline{B(w)}A(z)}{\pi(z-\overline w)}.
$$
Let $\Lambda\subset \CC$ be such that the system of reproducing kernels
$\{K_\lambda\}_{\lambda\in \Lambda}$ of the space $\he$ is exact.
Then there exists the generating function, that is, an 
entire function $G \in \he+z\he$, such that $G H\notin \he$ for any nontrivial entire function $H$, and  
vanishing exactly on the set $\Lambda$. 
The biorthogonal system to  $\{K_\lambda\}_{\lambda\in\Lambda}$ 
is given by 
$$
g_\lambda(z):= \frac{G(z)}{G'(\lambda)(z-\lambda)}.
$$

We will assume that $\{g_\lambda\}_{\lambda\in\Lambda}$ is also an exact system in $\he$
(recall that this is the case, e.g., when $\phi'\in L^\infty(\rl)$
\cite{fric} or when $\phi'$ has at most power growth and $\Theta=E^*/E$ 
has no finite derivative at $\infty$ \cite{bb}).

Denote by $T=\{t_n\}$ the zero set of $A$ (assume that $T\cap\Lambda=\emptyset$) 
and recall that the functions 
$$
\frac{A(z)}{z-t_n} = \pi i\frac{K_{t_n}(z)}{E(t_n)} 
$$
form an orthogonal basis in $\he$ \cite[Theorem 22]{br}
and $\big\|\frac{A(z)}{z-t_n}\big\|^2 =\pi\phi'(t_n)$. 
Then every $h \in \he$ can be written as
\begin{equation}
h(z)=A(z) \sum_n\frac{\overline a_n \mu_n^{1/2}}{z-t_n}, \qquad \{a_n\}\in\ell^2,
\label{dx2}
\end{equation}
where $\mu_n=1/\phi'(t_n)$, 
$$
\sum_n \frac{\mu_n}{1+t_n^2} <\infty.
$$

Let $h \in \he$ be orthogonal to 
$\{g_\lambda\}_{\lambda\in \Lambda_1} \cup \{K_\lambda\}_{\lambda\in \Lambda_2}$.
Then
\begin{equation}
\label{ext1}
\sum_n\frac{\overline a_n \mu_n^{1/2}}{z-t_n} = \frac{G_2(z)S_2(z)}{A(z)}
\end{equation}
for some entire function $S_2$. 
As in the Paley--Wiener case we assume that $G_2$
is an entire function which vanishes exactly on $\Lambda_2$ 
and $G_2^*/G_2 = B_1/B_2$ for some Blaschke products $B_1$ and $B_2$.
On the other hand, since $h\perp g_\lambda$, $\lambda\in 
\Lambda_1$,
we obtain 
\begin{equation}
\label{ext2}
\sum_n \frac{G(t_n)}{E(t_n)} \frac{a_n \mu_n^{1/2}}{z-t_n} =
i \frac{G_1(z)S_1(z)}{A(z)}
\end{equation}
for some entire function $S_1$ (argue as in the Paley--Wiener case).
Comparing the residues we get
\begin{equation}
S_1(t_n)G_1(t_n) = -i 
\frac{a_n \mu_n^{1/2}A'(t_n) G(t_n)}{E(t_n)},
\label{deBrange1}
\end{equation}
and
\begin{equation}
S_2(t_n)G_2(t_n) = \overline a_n \mu_n^{1/2} A'(t_n).
\label{deBrange2}
\end{equation}
Hence, for $S = S_1S_2$, we have
$$
S(t_n) = -i |a_n|^2 \mu_n (A'(t_n))^2/E(t_n).
$$
Since $A'(t_n) = (-1)^n |E(t_n)| \phi'(t_n)$
(the phase function $\phi$ is chosen in such a way that $\phi(t_n) = \pi/2+\pi n$), we get
\begin{equation}
\label{ext3}
S(t_n) 
=|a_n|^2 A'(t_n).
\end{equation} 

In what follows we need the following  theorem due to M.G.~Krein
(see, e.g., \cite[Chapter I, Section 6]{hj}):
{\it If an entire function $F$ is of bounded type both in $\mathbb{C}_+$
and in $\mathbb{C}_-$, then $F$ is of finite exponential type. 
If, moreover, $F$ is in the Smirnov class
both in $\mathbb{C}_+$ and in $\mathbb{C}_-$, then $F$
is of zero exponential type.}
Recall that a function $f$ analytic in $\mathbb{C}_+$ is said 
to be of {\it bounded type}, if
$f=g/h$ for some functions $g$, $h\in H^\infty(\mathbb{C}_+)$.
If, moreover, $h$ may be taken to be outer, we say that $f$
is in \textit{the Smirnov class in $\mathbb{C}_+$}.

In particular, any analytic function $f$ such that $\ima f>0$ in $\mathbb{C}_+$
is in the Smirnov class. In what follows we use the fact that 
if we put $\Theta = E^*/E$, then $\Theta$ is inner, and 
both $A/E = 1+\Theta$ 
and $E/A = (1+\Theta)^{-1}$ are in the Smirnov class.
Another useful observation is that  if $G$ is a generating function of some
exact system of reproducing kernels, then both $G/E$ and $G^*/E$
are of the form $Bh$, where $B$ is a Blaschke product and $h$ is outer
in $\mathbb{C}_+$.
Indeed, if $G/E$ has an exponential factor, i.e., $G(z)/E(z) = e^{iaz}
B(z)h(z)$, where $a>0$ and $h$ is
outer, then the function
$$
z\mapsto E(z)\frac{e^{iaz} - 1}{z} B(z)h(z)
$$ belongs to $\mathcal{H}(E)$
and vanishes at $\Lambda$.


From now on we assume that $\phi$ is of {\it tempered growth}, that is, 
\begin{equation}
\label{tempered}
\phi'(t) =O(|t|^N), \qquad |t| \to\infty, 
\end{equation}
for some $N$. It follows from (\ref{tempered}) that,
for any $F\in\he$, 
$$
\frac{|F(x)|}{|E(x)|} \le \frac{\|K_x\|_E\|F\|_E}{|E(x)|} =
\Big(\frac{\phi'(x)}{\pi}\Big)^{1/2} \|F\|_E \lesssim (|x|+1)^{N/2}, \quad x\in\rl.
$$
Using the same arguments as in the proof of Lemma~\ref{g1s1inpw} 
we get $G_1S_1\in \he + z\he$. Hence,
\begin{equation}
\label{ext0}
GS
\in 
\mathcal{P}_{\frac{N}{2}+1} \cdot \mathcal{H}(E^2),
\end{equation}
where $\mathcal{P}_M$ is the set of polynomials of degree at most $M$.

Arguing analogously to the proof of Lemma~\ref{growth1}
we obtain the following growth restriction.

\begin{lemma} \label{growth2}
Assume that $\phi$ satisfies \eqref{tempered}.
Let $h\in\he$ be orthogonal to some system 
$\{g_\lambda\}_{\lambda\in \Lambda_1} \cup \{K_\lambda\}_{\lambda\in \Lambda_2}$
and let $(S_1, S_2)$ be the corresponding pair.
Then $S \in  \mathcal{P}_M \cdot \mathcal{H}(E)$ 
for some $M$ depending only on $N$ and
\begin{equation}
\label{as2}
\bigg|\frac{S(iy)}{A(iy)}\bigg| \gtrsim \frac{1}{|y|^{K}}, \qquad |y|\to\infty,
\end{equation}
for some $K>0$. 
\end{lemma}

\begin{proof}
By \eqref{ext3} we have 
$$
\frac{|S(t_n)|}{|E(t_n)| (\phi'(t_n))^{1/2}} = |a_n|^2
(\phi'(t_n))^{1/2} \lesssim |a_n|^2 |t_n|^{N/2},
$$
and, dividing out sufficiently many zeros $s_1, \dots, s_M$ of $S$ we obtain that
$$
\sum_n \frac{|\tilde S(t_n)|^2}{|E(t_n)|^2 \phi'(t_n)} <\infty, \qquad
\tilde S(z) = \frac{S(z)}{(z-s_1)\cdots (z-s_M)}.
$$
Now let $Q$ be the (unique) function in $\he$ which solves
the interpolation problem $Q(t_n) = \tilde S(t_n)$.
Using (\ref{ext0}) and an analogous estimate for $GQ$,
we obtain that 
$G(\tilde S -Q) \in \mathcal{P}_M\cdot \mathcal{H}(E^2)$.
Since $\tilde S - Q$ vanishes on $\{t_n\}$, we have 
$G(\tilde S -Q) = GAH \in \mathcal{P}_M\cdot \mathcal{H}(E^2)$
for some entire function $H$. We want to show that $H$
is a polynomial of degree at most $M+1$, 
whence $\tilde S = Q+AH\in \mathcal{P}_{M +1}\cdot \mathcal{H}(E)$.

By the remarks after the formulation of Krein's theorem,   
$(GA)/E^2$ and $(G^*A)/E^2$ are of
the form $Bh$,
where $B$ is a Blaschke product and $h$ is outer  in $\mathbb{C}_+$. Since
$GAH = g \in \mathcal{P}_M \cdot \mathcal{H}(E)$,
we see that $H = \frac{g}{E^2} \cdot \frac{E^2}{GA}$ is in the Smirnov class
in $\mathbb{C}_+$ and the same holds for $H^*$. Then, by
Krein's theorem, $H$ is of zero exponential type.


If $H$ has at least $M+2$ zeros, then dividing them out
we obtain an entire function $\tilde H$ such that 
$GA\tilde H \in \mathcal{H}(E^2)$ and $|G(iy)\tilde H(iy)|/|E(iy)| = o(y^{-1})$, 
$|y|\to\infty$ (we use the fact that $|A(iy)|/|E(iy)| \gtrsim y^{-1}$,
$y\to +\infty$). Let $v_n$ be such that $\phi(v_n) =  \pi n$
(thus, $\{v_n\}$ is the support of another orthogonal family 
of reproducing kernels). Since $|A(v_n)| = |E(v_n)|$,
we conclude that 
$$
G(v_n)\tilde H(v_n)/E(v_n) \in L^2(\nu),
\qquad \nu = \sum_n (\phi'(v_n))^{-1} \delta_{v_n}. 
$$
Now it remains to apply \cite[Theorem 26]{br} to conclude that 
$G\tilde H\in \he$, a contradiction to the fact that $G$
is the generating function of a complete system of kernels.

We have shown that $S=H_1(Q + AH_2)$ for some polynomials $H_1,H_2$.
It follows from the representation of functions in $\mathcal{H}(E)$
(formula \eqref{dx2}) that $Q(iy) + A(iy)H_2(iy) \sim A(iy)H_2(iy)$
for any $Q\in \mathcal{H}(E)$ and any nonzero polynomial $H_2$.
Thus, in this case $|S(iy)| \gtrsim |A(iy)|$, $|y|\to \infty$,
and (\ref{as2}) is trivial.
In the case when $H_2 \equiv 0$ and $S=H_1 Q$
we use that the function $Q$ is the solution of the interpolation problem
$$
\begin{aligned}
Q(t_n) & =  \frac{S(t_n)}{(t_n-s_1)\cdots (t_n-s_M)} \\
& = 
\frac{A'(t_n) |a_n|^2}{(t_n-s_1)\cdots (t_n-s_M)} = 
A'(t_n) |a_n|^2 \bigg(\frac{1}{t_n^M} + \frac{b_n+ic_n}{t_n^{M+1}}\bigg),
\end{aligned}
$$
where $\{b_n\}_n$ and $\{c_n\}_n$ are bounded sequences, and assume
without loss of generality that $M$ is even and 
$t_n \ne 0$. Then 
$$
\frac{Q(z)}{A(z)} = \sum_n \frac{|a_n|^2}{z-t_n}
\bigg(\frac{1}{t_n^M} + \frac{b_n+ic_n}{t_n^{M+1}}\bigg),
$$
and
$$
-\ima \frac{Q(iy)}{A(iy)}  = 
\sum_n \frac{|a_n|^2}{y^2 + t_n^2} \bigg(\frac{y}{t_n^M} +\frac{b_n y}{t_n^{M+1}} + 
\frac{c_n}{t_n^{M}}\bigg).
$$
All the sums in the brackets except, possibly, 
a finite number are positive when $y\to +\infty$ and 
negative when $y\to-\infty$. Expanding the right-hand side in powers of $1/y$, we
deduce \eqref{as2}. 
\end{proof}

It follows from (\ref{as2}) that $S^*/S$ is a ratio of two Blaschke
products, i.e., has no exponential factor.
We show now that the same is true for each of the functions $S_2^*/S_2$ and $S_1^*/S_1$. 
Suppose that $G_2^*S_2^*/(G_2S_2)$ is not a ratio
of Blaschke products, i.e., let
$G_2^*S_2^*/(G_2S_2) = e^{ibz}B_1/B_2$, where $B_1$ and $B_2$ are meromorphic Blaschke products
and $b\in \mathbb{R}$. Assume that $b>0$ (the case $b<0$ is analogous). Then
the function $e^{i c z} S_2G_2$, $0<c\le b$, is also in $\mathcal{H}(E)$
and formulas \eqref{ext1} and \eqref{ext2} will hold also for
the functions  $e^{i c z}S_2G_2$  and  $e^{-icz}S_1G_1$, $0<c < b$,
with $\{e^{-i c t_n}a_n\}_n$ in place of $\{a_n\}_n$. Hence, 
$(S_1e^{- i c z},S_2e^{i c z})\in\Sigma(\Lambda_1, \Lambda_2)$ and 
$((1+e^{-icz})S_1,(1+e^{icz})S_2)\in\Sigma(\Lambda_1, \Lambda_2)$.
Now, by Lemma~\ref{growth2},  
the function $\tilde S(z) = S(z)(1+e^{i c z})(1+e^{-i c z})$ 
belongs to $\mathcal{P}_M \cdot \mathcal{H}(E)$, whence
$\tilde S/A$ is of Smirnov class in the upper half-plane.
However, this contradicts to \eqref{as2}.
Thus, $S_2/S^*_2$ and $S_1/S^*_1$ are ratios of Blaschke products.


Now by an argument, analogous to that in the proof of Lemma~\ref{star}, the pair
$(S_1^*, S_2^*)$ also corresponds to some function orthogonal to
$\{g_\lambda\}_{\lambda\in \Lambda_1} \cup \{K_\lambda\}_{\lambda\in \Lambda_2}$.
Thus, we may always find functions $S_1,S_2$ which are real on $\rl$.
By \eqref{ext3}, the function $S$ changes its sign at adjacent points $t_n$
(as usual we assume that the basis is chosen in such a way that 
all coefficients $a_n$ are nonzero), and thus,
there is a zero of $S$ in each of the intervals $(t_n, t_{n+1})$.
We have an analog of Corollary \ref{zeros1}.

\begin{lemma}
\label{zeros2}
Assume that $\phi$ satisfies \eqref{tempered}.
If a pair $(S_1, S_2)$ corresponds 
to a function $h \in \he$ orthogonal to some system 
$\{g_\lambda\}_{\lambda\in \Lambda_1} \cup \{K_\lambda\}_{\lambda\in \Lambda_2}$
and $S_1$ and $S_2$ are real on $\rl$, then
$S = S_0 H$, where $S_0$ has exactly one zero in any interval
$(t_n, t_{n+1})$ and $H$ is a polynomial of degree bounded by $M=M(N)$.
\end{lemma}

\subsection{Proof of Theorem \ref{main3}}
Without loss of generality assume that $\phi$ is unbounded both from below and from above, and $\Lambda \cap \{t_n\} =\emptyset$,
where $\phi(t_n) = \pi n$, $n\in\zl$.
Let $f$ and $h$ be orthogonal to the system (\ref{syst3}),
$$
f(z)=A(z) \sum_n\frac{\overline a_n \mu_n^{1/2}}{z-t_n}, 
\qquad
h(z)=A(z) \sum_n\frac{\overline b_n \mu_n^{1/2}}{z-t_n}, 
\qquad \{a_n\}, \ \{b_n\} \in\ell^2.
$$
Let $(S_1, S_2)$ and $(T_1, T_2)$  be the corresponding pairs
of entire functions such that $S_1$, $S_2$, $T_1$ and $T_2$ 
are real on $\rl$.
Using the equations \eqref{deBrange1}--\eqref{deBrange2} 
in the same way as in the proof of Theorem \ref{main1}, we obtain
$$
S_1(t_n)T_2(t_n) = T_1(t_n)S_2(t_n)  = a_n b_n |E(t_n)| \phi'(t_n) \beta_n,
$$
where $|\beta_n| =1$. The hypothesis $\sup_n |\phi'(t_n)|<\infty$
implies that
$$
\sum_n \frac{|S_1(t_n)T_2(t_n)|^2}{|E(t_n)|^2\phi'(t_n)} = 
\sum_n a_n^2b_n^2
<\infty.
$$
Since $\{K_{t_n}\}$ is an orthogonal basis in $\he$
and $\|K_{t_n}\|_E^2 = |E(t_n)|^2 \phi'(t_n)/\pi$,
we conclude that there exists a unique function $Q\in \he$
which solves the interpolation problem
$Q(t_n) = a_n b_n |E(t_n)| \phi'(t_n) \beta_n$. Then 
$$
T_1(z) S_2(z) = Q(z) +a(z)A(z), \qquad 
S_1(z) T_2(z) = Q(z) +b(z)A(z),
$$  
for some entire functions $a$ and $b$.
We show now that $a$ and $b$ are polynomials. 

Note that by Lemma~\ref{growth2} 
the functions $S=S_1S_2$ and $T=T_1T_2$ 
as well as $(S_1+T_1)(S_2+T_2)$ and $(S_1+iT_1)(S_2-iT_2)$
are in $\mathcal{P}_M \cdot \mathcal{H}(E)$. 
Hence, the functions
$S_1T_2$ and $S_2T_1$, and, consequently,
the functions 
$S_1T_2-Q$ and $S_2T_1 -Q$ are in $\mathcal{P}_M \cdot \mathcal{H}(E)$.

Now assume that $F=AH$ for some entire function $H$, and 
$F\in \mathcal{P}_M \cdot \mathcal{H}(E)$.
First, since $E/A$ and $F/E$ are in the Smirnov class in $\mathbb{C}_+$,
we conclude that, by Krein's theorem, $H$ is of zero exponential type. 
We claim that $H$ must be a polynomial.

Indeed, if $H$ has at least $M$ zeros $z_j$, then dividing 
$F$ by $\prod_{j=1}^M (z-z_j)$ we obtain a function in $\he$ 
which vanishes on $\{t_n\}$ and, thus, is identically zero.
Applying this argument to $S_1T_2-Q$ and $S_2T_1 -Q$
we conclude that $a$ and $b$ are polynomials. 

Now assume that $a \ne 0$.
Let us denote by $s_m$ the zero of $S_2$ 
such that $|\phi(s_m) - \phi(t_m)| \le\pi/2$
whenever such a zero exists. Then
$$
Q(s_m) + a(s_m)  A(s_m) = 0.
$$
Note that $\{t_m\}$ is separated sequence (i.e., $\inf_{n\ne m}|t_n-t_m|>0$)
and so $s_m$ is the union of two separated sequences. 
By a simple variant of Carleson embedding theorem for the de Branges spaces
with $\phi' \in L^\infty(\rl)$ (an explicit statement 
may be found in \cite[Theorem 5.1]{bar00}, 
though the proof may be recovered already from \cite[Theorem 2]{vt}) 
we have
$$
\sum_m \frac{|Q(s_m)|^2}{|E(s_m)|^2}<\infty
$$
for any $Q\in \he$, whence 
$$
\sum_m \frac{|A(s_m)|^2}{|E(s_m)|^2}<\infty.
$$
By the definition of the phase function, $|A(s_m)| = 
|E(s_m) \sin(\phi(s_m) - \phi(t_m))|$.
Thus, we obtain that 
$$
\sum_m \sin^2(\phi(s_m) - \phi(t_m))
\asymp \sum_m (\phi(s_m) - \phi(t_m))^2  <\infty
$$

To complete the proof we apply once again 
the argument with the shift of the basis.
The zeros of $S_2$ do not depend on the choice of the basis.
Expanding with respect to another basis, say $\{K_{\tilde t_n}\}$, 
with $\phi(\tilde t_n) = \delta +\pi n$ for some small $\delta$,
we get that $\sum_m (\phi(s_m) - \phi(\tilde t_m))^2 < \infty$.
However, $|\phi(t_m) - \phi(\tilde t_m)| = \delta$
and we come to a contradiction.

Thus, we have proved that $a = b =0$, and so $S_1 T_2 = T_1 S_2 = Q$.
Since $S_1$ has no common zeros with $S_2$ (we choose the basis so that all $a_n$
are nonzero) we conclude that the zero sets of $S_2$ and  $T_2$ coincide, 
and, thus, $f$ is proportional to $h$.

\begin{remark}
{\rm  It is easy to show that if 
$\phi$ is of tempered growth, then the orthogonal complement 
to the system (\ref{syst3}) is always finite dimensional, with a bound
on the dimension depending only on $N$. 
Indeed, by Lemma~\ref{zeros2}, there exists $M = M(N)$ such that 
for any pair $(S_1, S_2)$ which corresponds to 
a function $f$ in the orthogonal complement to (\ref{syst3})
and is real on $\RR$,  
we have $S=S_0 H$, $H\in \mathcal{P}_M$. In particular, any interval 
$(t_n, t_{n+1})$ contains at most $M+1$ zeros of $S$.

Now assume that the orthogonal complement to (\ref{syst3}) contains at least 
$M+3$ linearly independent vectors $f_{j,0}$, $j=1, \ldots M+3$, such that
the corresponding functions $S_{1,j,0}$, $S_{2,j,0}$ are real on $\RR$. 
Considering linear combinations (with real coefficients)
$f_{j,1}= f_{j,0}- \alpha_j f_{M+3,0}$, $j=1, \ldots, M+2$, 
we may achieve that the functions $S_{1,j,1}$ 
corresponding to $f_{j,1}$ have a common zero at $x_1\in (t_0,t_1)$.
Repeating this procedure we obtain a nonzero 
function $f_{M+2,1}$ in the orthogonal complement to (\ref{syst3}) such that
the corresponding function $S_{1,M+2,1}$ vanishes at 
$M+2$ distinct points $x_1, \dots x_{M+2} \in (t_0, t_1)$ which gives a contradiction.}
\end{remark}

\subsection{Density results}
Let a pair $(S_1, S_2)$ correspond 
to a function $h \in \he$ orthogonal to some system 
$\{g_\lambda\}_{\lambda\in \Lambda_1} \cup \{K_\lambda\}_{\lambda\in \Lambda_2}$
and let $S_1$ and $S_2$ be real 
on $\rl$. We show that most of the zeros of $S$ are in a certain sense 
close to the set $\{t_n\}$ (the support of a de Branges orthogonal basis).
Thus, the zeros of $S_2$ which do not depend on the choice
of the basis form a small proportion of the zeros of $S$ 
(see Corollary \ref{dens1} below).
                         
By Lemma~\ref{zeros2}, $S=S_0 H$,
where $S_0$ has exactly one zero in each of the intervals $(t_n, t_{n+1})$
and $H$ is a polynomial. Moreover, by (\ref{ext3})
we have $\{S(t_n)/ A'(t_n)\} \in\ell^1$, whence
$\{S_0(t_n)/ A'(t_n)\} \in\ell^1$.
By Lemma~\ref{growth2} we have $S \in \mathcal{P}_M\cdot\he$ 
for some $M$, whence $S_0/A$ grows at most polynomially along 
$i\mathbb{R}_+$.
Since the zeros of $A$ and $S_0$ interlace, the function
$S_0/A$ is a Herglotz function and thus has a representation
\begin{equation}
\label{herg}
\frac{S_0(z)}{A(z)} =az+b +\sum_n \frac{c_n}{z-t_n}, \qquad 
\{c_n\}\in\ell^1.
\end{equation}
We will show that in this case the zeros of $S_0$ (and $S$) 
must be necessarily close (in some sense) 
to the points $t_n$. The case when $a\ne 0$ or $b\ne 0$ 
should be treated exactly as in Proposition~\ref{closeness}. 
The remaining case follows from the following 
proposition (apparently, known to experts).

\begin{proposition}
\label{closeness1}
Let $t_n\in \RR$, $n\in\ZZ$, $t_n \to \pm \infty$, $n\to \pm \infty$,
and let $\mu_n>0$, $\sum_n \mu_n = M<\infty$. Let $A$ 
be an entire function which is real on $\RR$ 
and has only simple real zeros at the points $\{t_n\}$.
Define an entire function $B$ by the Herglotz representation 
$$
\frac{B(z)}{A(z)} = \sum_n \frac{\mu_n}{z- t_n}.
$$
Denote by $s_n$ the zero of $B$ in $(t_n, t_{n+1})$. Then
\begin{equation}
\label{asimp0}
\sum_{s_n>0} \frac{t_{n+1} - s_n}{s_n}<\infty, \qquad\quad
\sum_{s_n<0} \frac{s_n - {t_n}}{|s_n|}<\infty.
\end{equation}
\end{proposition}

\begin{proof}
The zeros of $B$ are simple and interlace with 
the zeros of $A$. Since $\ima \frac{B}{A} > 0$ in $\cp$, the function
$E=A-iB$ is in the Hermite--Biehler class and so 
we can define the de Branges space $\he$. The measure $\mu = \sum_n \mu_n \delta_{t_n}$ 
is a corresponding Clark measure for which the embedding  
operator $\frac 1{\pi E}\he \to L^2(\mu)$ is unitary.

Consider the inner function $\Theta = E^*/E$. 
Since $2A/E = 1+\Theta$ 
and $2B/E = -i(\Theta-1)$, we have  
$$
i\frac{1-\Theta(z)}{1+\Theta(z)} = \int_\RR \frac{d\mu(t)}{t-z} 
\sim i\frac{M}{y}, \qquad z=iy, \ y\to+\infty.
$$
Hence, 
\begin{equation}
\label{asimp1}
\frac{1+\Theta(iy)}{1-\Theta(iy)} \sim \frac{y}{M}, \qquad y\to+\infty.
\end{equation}
It is well known 
that the function $\Theta$ may be reconstructed from 
the sets $\{t_n\}=\{\Theta = 1\}$ and $\{s_n\}=\{\Theta = -1\}$
by the formula 
$$
\log\frac{\Theta+1}{\Theta-1} = c +
\int_\RR
\bigg(\frac{1}{t-z}-\frac{t}{t^2+1} \bigg) f(t) dt,
$$
where
$$
f(t) = \begin{cases}
- 1/2, & t\in (t_n, s_n), \\
  1/2, & t\in (s_n, t_{n+1}),
\end{cases}
$$
and $c\in \RR$ (essentially, this is a very 
special case of the Krein spectral shift formula \cite{kr}, see also 
\cite[Section 6.1]{pm}). 
Then, by (\ref{asimp1}), we have 
$$
\int_\RR \frac{(1-y^2) t}{(t^2+y^2)(t^2+1)} f(t) dt =
\rea \int_\RR \bigg(\frac{1}{t-iy}-\frac{t}{t^2+1} \bigg) f(t) dt =
\log y +O(1), 
\quad y\to +\infty.
$$
A direct computation shows, however, that 
$$
\int_\RR \frac{(y^2-1) |t|}{(t^2+y^2)(t^2+1)} |f(t)| dt =
\log y +O(1), \qquad y\to +\infty,
$$
whence
$$
\int_{\{t:\  t f(t)>0 \}} 
\frac{(y^2-1)\, tf(t)}{(t^2+y^2)(t^2+1)}  dt = O(1), \qquad y\to +\infty,
$$
and therefore 
$$
\int_{\{t: \ t f(t) > 0 \}} 
\frac{t f(t)}{t^2+1} dt <\infty. 
$$
Since $tf(t)> 0$ for $t\in (s_n, t_{n+1})$, $s_n>0$, 
or $t\in (t_n, s_n)$, $s_n<0$, we have
$$
\sum_{s_n>0} \int_{s_n}^{t_{n+1}} \frac{dt}{t}= \sum_{s_n>0} \ln\frac{t_{n+1}}{s_n}
<\infty, \qquad
\sum_{s_n<0} \int_{t_n}^{s_n} \frac{dt}{|t|}
=\sum_{s_n<0} \ln\frac{|t_n|}{|s_n|}<\infty. 
$$
The latter convergences are obviously equivalent to (\ref{asimp0}).
\end{proof}

As a corollary we immediately obtain a slightly refined version
of Proposition~\ref{closeness}. Moreover, if $t_n = n$, 
$n\in\ZZ$, $A(z)=\sin \pi z$, and $S=S_1S_2$ is the function
arising  from the possible one-dimensional defect in the Paley--Wiener space,
then
$$
\sum_{s\in \mathcal{Z}_2} \frac{1}{|s|} <\infty.
$$
Indeed, the zero set $\mathcal{Z}_2$ of the 
function $S_2$ does not depend on the choice
of the basis, therefore applying Proposition~\ref{closeness1} 
to $t_n = n$ and  $t_n = n+\delta$ (e.g., $\delta = \frac{1}{2}$), 
$n\in\ZZ$, we obtain
$$
\sum_{s\in \mathcal{Z}_2, \, s>0} \frac{[s]+1-s}{s} <\infty, \qquad 
\sum_{s\in \mathcal{Z}_2, \, s>0} \frac{[s-\delta] + 1+\delta-s}{s} <\infty.
$$

Under natural regularity conditions, Proposition~\ref{closeness1} 
implies the following 
closeness of the sequences $\{t_n\}$ and $\{s_n\}$.

\begin{corollary}
\label{dens1}
Let $A$, $B$, $\{t_n\}$ and $\{s_n\}$ be as in Proposition~\ref{closeness1}.
Put $I_n = [t_n, t_{n+1}]$. Assume that $|I_k| \asymp |I_n|$, $n\le k\le 2n$,
with the constants independent on $k, n$, and that $|t_{an}| \ge \rho |t_n|$
with some $a\ge 2,\rho>1$. Then for any $\delta>0$ the set $\mathcal{N}$ of indices 
$n$ such that $t_n>0$ and $t_{n+1}-s_n\ge \delta |I_n|$ 
$($respectively, $t_n<0$ and $s_n-t_n \ge  \delta |I_n|$$)$ has zero 
density. 
\end{corollary}

\begin{proof}
Note that $|t_k|\asymp |t_n|$, $n\le k \le an$.
If the upper density of $\mathcal{N}$ is positive, then there exists 
a sequence $M_j\to\infty$ such that 
$$
\sum_{n\in [M_j, a M_j]\cap\mathcal{N}} \frac{t_{n+1} - s_n}{s_n} \gtrsim 
\sum_{n\in [M_j, a M_j]\cap\mathcal{N}} \frac{t_{n+1} - t_n}{t_n}
\gtrsim \sum_{n\in [M_j, a M_j]} \frac{t_{n+1} - t_n}{t_n} \gtrsim 
\log \frac{t_{aM_j}}{t_{M_j}} \ge \log \rho,
$$
and the first series in (\ref{asimp0}) diverges, a contradiction.
\end{proof}

Arguing as in the proof of Theorem~\ref{main2} we deduce 
Theorem~\ref{dop7} from Corollary~\ref{dens1}.


\subsection{Nonhereditarily complete systems of reproducing 
kernels in de Branges spaces}
In this section we prove Theorem \ref{example}, i.e.,
we construct a de Branges space $\he$
and a complete and minimal system of reproducing kernels 
$\{K_\lambda\}_{\lambda\in \Lambda}$ such that its
biorthogonal system is also complete, but the system 
$\{K_\lambda\}_{\lambda\in \Lambda}$ is not hereditarily complete.

We have already seen 
that the existence of a nonhereditarily complete system 
of reproducing kernels generated by some function $G$ 
in the de Branges space $\he$ 
is equivalent to the solvability of the equations
$$
\sum_n\frac{\overline a_n \mu_n^{1/2}}{z-t_n} 
=\frac{G_2(z)S_2(z)}{A(z)},
$$
\begin{equation}
\label{ext2c}
\sum_n \frac{G(t_n)}{E(t_n)} \cdot \frac{a_n \mu_n^{1/2}}{z-t_n} =
i \frac{G_1(z)S_1(z)}{A(z)}
\end{equation}
for some nonzero  $\{a_n\}\in\ell^2$ and 
some entire functions $S_1$  and $S_2$.
If all the above objects are found, then $h=G_2S_2$ 
is orthogonal to the corresponding system.
The corresponding equations will be constructed 
as small perturbations of an orthogonal expansion in a de Branges space
with respect to a reproducing kernels basis.

Let the sequence $\{t_n\}$ satisfy (\ref{hypot}). 
Without loss of generality we may assume that $t_n \ge 0$, $n\ge 0$
and $t_n <0$, $n<0$.
It follows from (\ref{hypot}) that $|t_n| \asymp |t_{n+1}|$ and $|t_n| 
\gtrsim |n|^\gamma$, $|n|\to\infty$, with some $\gamma >0$.

We construct the space $\he$ and the functions $G_1$, $G_2$, $S_1$
and $S_2$ in the reverse order. Namely, we start with the construction of 
the function $S$. 
First choose two sequences of positive integers 
$n_k, l_k \to \infty$ with the following properties:
$$
2t_{n_k} < t_{n_k + l_k} < \frac{t_{n_{k+1}}}{2} \quad
\text{and} \quad 
k(t_{n_k + 1}-t_{n_k}) \le t_{n_k}/100, 
\qquad k\in \mathbb{N}.
$$

Let $a_n \in \RR$ be such that 
$$
|a_{n_k}| = |a_{n_{k}+1}| = |a_{n_k +l_k}| = |a_{n_{k}+ l_k +1}| = k^{-1},
$$
and let $|a_n| = (|n|+1)^{-1}$ for all other values of $n$.
Note that $|a_n| \gtrsim |t_n|^{-M}$ for some $M>0$.
The signs of $a_n$ will be specified later on. 
Let $A$ be a canonical Hadamard product (of finite genus) whose zeros are simple
and coincide with $\{t_n\}$ (thus, 
$A$ is real on $\rl$). Define the entire function $S$ by
$$
\frac{S(z)}{A(z)} = \sum_{n} \frac{a_n^2}{z-t_n}.
$$
Then $S$ has exactly one zero $z_n$ in each interval $(t_n,t_{n+1})$. 

We write $S$ as the product
$$
S=S_1S_2 = T_0T_1S_2, 
$$
where $T_0$ is the canonical product with the zeros $s_k = z_{n_k}$ in intervals 
$(t_{n_k}, t_{n_{k}+1})$ and $S_2$ is a canonical product with the zeros 
$z_{n_k + l_k}$ in $(t_{n_k + l_k}, t_{n_k + l_k +1})$, $k \in \mathbb{N}$.
Next we construct $h$. We will construct it as $h=\tilde T_0 T_1 S_2$
where $\tilde T_0$ is a perturbation of the function $T_0$
such that
\begin{equation}
\label{h1}
\frac{h(z)}{A(z)} = 
\sum_{n} \frac{c_n |a_n|}{z-t_n}, 
\end{equation}
\begin{equation}
\label{h2}
\sum_{n} c_n^2=\infty,
\qquad \sum_{t_n\ne 0} \frac{c_n^2}{t_n^2}<\infty.
\end{equation}
Condition (\ref{h1}) means that 
$$
\frac{S(z)}{A(z)}\cdot\frac{\tilde T_0(z)}{T_0(z)} = 
\sum_{n} \frac{\tilde T_0(t_n)}{T_0(t_n)}\cdot
\frac{a_n^2}{z-t_n}, 
$$
and $c_n= |a_n|\tilde T_0(t_n)/T_0(t_n)$. 
Let us show that all these conditions 
may be  satisfied. 

Assume that $|s_k -t_{n_k}|>|s_k- t_{n_k +1}|$. Then 
we shift the zero $s_k$ of $T_0$ in the following way: 
$$
\tilde s_k = t_{n_k+1} - k |s_k- t_{n_k +1}| \rho_k.
$$
(Analogously, if $|s_k -t_{n_k}|\le |s_k- t_{n_k +1}|$, we put
$$
\tilde s_k = t_{n_k} - k |s_k- t_{n_k}| \rho_k;
$$
in what follows we consider only the first situation.)
Let $\tilde T_0$ be the canonical product with the zeros $\tilde s_k$.

By hypothesis (\ref{hypot}) we may choose $\rho_k \in(1, 2)$ such that 
\begin{equation}
\label{h20}
\dist(\tilde s_k,\{t_n\}_{n\ne n_{k} +1}) \gtrsim |t_{n_k}|^{-N}
\end{equation}
for some $N>0$, 
$\tilde s_k \in (t_{n_k+1}/2, t_{n_k+1})$ 
and zero sets of $\tilde T_0$ and $T_1S_2$ do not intersect.
An easy estimate of the infinite products shows that
with such choice of zeros for $\tilde T_0$ we have 
$$
\bigg|\frac{\tilde T_0(x)}{T_0(x)}\bigg| 
\asymp \bigg|\frac{x-\tilde s_k}{x-s_k}\bigg|, \qquad
x \in \Big( \frac{t_{n_k} + t_{n_{k-1}}}{2}, 
\frac{t_{n_k} + t_{n_{k+1}}}{2}  \Big).
$$
Then we obtain 
$$
|c_{n_k+1}| \asymp \bigg|
\frac{\tilde T_0(t_{n_k+1})}{T(t_{n_k+1})}\bigg| \cdot |a_{n_k+1}|
\asymp
\bigg|\frac{t_{n_k+1}-\tilde s_k}{t_{n_k+1}-s_k}\bigg| \cdot  k^{-1}  \asymp 1,
$$
whence the first series in (\ref{h2}) diverges.
Moreover, it is easy to see that 
$$
\bigg| \frac{\tilde T_0(t_n)}{T_0(t_n)}\bigg| \gtrsim 1, \qquad
t_n \in \Big[t_{n_{k-1}}, \frac{t_{n_k}}{2}\Big] \cup [t_{n_k+1}, t_{n_{k+1}}],
$$
while
$$
\bigg| \frac{\tilde T_0(t_n)}{T_0(t_n)}\bigg| \gtrsim 
\frac{\dist(\tilde s_k,\{t_n\}_{n\ne n_{k} +1})}{t_n}, \qquad
t_n \in \Big[\frac{t_{n_k}}{2}, t_{n_k}\Big].
$$
Thus, by (\ref{h20}), we have
\begin{equation}
\label{imp}
|t_n|^{-N-1} \lesssim 
\bigg| \frac{\tilde T_0(t_n)}{T_0(t_n)}\bigg| \lesssim k, \qquad 
n_{k-1} \le n \le n_k,
\end{equation}
and $\Big| \frac{\tilde T_0(t_n)}{T_0(t_n)}\Big|\asymp 1$ 
for $n \le 0$.
Hence, 
\begin{equation}                          
\label{imp3}
\begin{aligned}
  t_n^{-N-1}|a_n|  \lesssim |c_n| & \lesssim |a_n| k \lesssim 1, 
  \qquad n_{k-1} \le n \le n_k, \\
  |c_n| & \asymp |a_n|, \qquad n\le 0,
\end{aligned}
\end{equation} 
and, thus, the second condition in 
(\ref{h2}) is satisfied (note that $k = o(t_{n_k})$, $k\to \infty$).

Moreover, $|\tilde T_0(iy)/T_0(iy)| \asymp 1$, and so both terms in (\ref{h1})
tend to zero along $i\RR$. We conclude that the interpolation formula holds.

Next we introduce a de Branges space $\he$. Put $\mu_n =c_n^2$
and $\mu = \sum_{n} \mu_n \delta_{t_n}$.
By (\ref{h2}), $\int (1+t^2)^{-1} d\mu(t) <\infty$, and
we can define a meromorphic inner function $\Theta$
by the formula
$$
\frac{1-\Theta(z)}{1+\Theta(z)} = 
\frac{1}{i}\int \bigg(\frac{1}{t-z} -\frac{t}{t^2+1} 
\bigg)d\mu(t), \qquad z\in \mathbb{C}_+.
$$
Then $\Theta = E^*/E$ for some entire function $E$ 
in the Hermite--Biehler class. We may assume that $E$ does not vanish on 
$\mathbb R$. Moreover, since the zero set of 
$E+E^*$ coincides with $\{t_n\}$, we may choose $E$ so that
$E+E^* = 2A$. Now, if we choose the signs of $a_n$
so that ${\rm sign}\, a_n = {\rm sign}\, c_n$, formula (\ref{h1})
becomes
$$
\frac{h(z)}{A(z)} = 
\sum_{n} \frac{a_n|c_n|}{z-t_n}=
\sum_{n} \frac{a_n \mu_n^{1/2}} {z-t_n}. 
$$
Hence, $h\in \he$. 

We have $h = \tilde T_0 T_1 S_2$. Put $G_2 = \tilde T_0 T_1$. Then $h=G_2S_2$
and it remains to construct $G_1$ so that $G$ is the generating function
of a complete and minimal system of reproducing kernels in 
$\he$ and (\ref{ext2c}) is satisfied.

We will construct $G_1$ as a small perturbation of $S_2$ as we did above.
We need to satisfy $G\notin \he$, $G\in \he+z\he$ and (\ref{ext2c})
which is rewritten as 
\begin{equation}
\label{ext2a}
\frac{S(z)}{A(z)}\cdot \frac{G_1(z)}{S_2(z)} = 
-i \sum_{n} 
\frac{G_1(t_n)}{S_2(t_n)}\cdot\frac{h(t_n)}{E(t_n)} \cdot 
\frac{ a_n \mu_n^{1/2}}{z-t_n}.
\end{equation}
Note that in any de Branges space we have 
$i A'(t_n) =  -E(t_n) \varphi'(t_n) =-E(t_n) \mu_n^{-1}$. 
Then (\ref{ext2a}) simplifies to
$$
\frac{S(z)}{A(z)} \cdot \frac{G_1(z)}{S_2(z)} = 
\sum_{n} 
\frac{G_1(t_n)}{S_2(t_n)}\cdot\frac{h(t_n)}{ A'(t_n) |c_n|}\cdot \frac{a_n}{z-t_n}.
$$
The residues, obviously, coincide.

Applying the above construction to $S_2$ in place of $T_0$ 
(i.e., shifting the zeros $z_{n_k+l_k}$) 
we construct $G_1$ (again we may assume that $G_1$ has no common zeros with 
$\tilde T_0 T_1$) so that 
\begin{equation}
\label{imp22}
\bigg|\frac{G_1(t_n)}{S_2(t_n)}\bigg| \lesssim k,
\qquad  l_k+n_k \le n \le n_{k+1} +  l_{k+1}.
\end{equation}
and
$$
\bigg|\frac{G_1(t_{n_k+l_k+1})}{S_2(t_{n_k +l_k+1})}\bigg| \cdot |a_{n_k+l_k+1}| 
\asymp 1. 
$$

Note that $|h(t_n)| = |A'(t_n)|\cdot |a_n| \cdot \mu_n^{1/2} = 
|E(t_n)|\cdot|a_n| \cdot|c_n|^{-1}$. Then 
$$
\bigg|\frac{G(t_n)}{E(t_n)}\bigg| = 
\bigg|\frac{G_1(t_n)}{S_2(t_n)}\bigg|\cdot|a_n| \cdot|c_n|^{-1}.
$$
Hence, in particular, 
$$
\bigg|\frac{G(t_{n_k+l_k+1})}{E(t_{n_k+l_k+1})}\bigg|
\asymp  |c_{n_k+l_k+1}|^{-1},
$$
whence, $\|G/E\|^2_{L^2(\mu)} = \sum_n |G(t_n)|^2 |E(t_n)|^{-2} |c_n|^{2} = \infty$. 
Thus, $G \notin \he$. However, by (\ref{imp22}),
$$
\sum_{t_n\ne 0} \frac{|G(t_n)|^2c_n^2} {t_n^2 |E(t_n)|^2} \lesssim 
\sum_{t_n\ne 0} \frac{a_n^2}{t_n^2}\bigg|\frac{G_1(t_n)}{S_2(t_n)}\bigg|^2
<\infty,
$$
whence $\frac{G(z)}{(z-\lambda)E(z)} \in L^2(\mu)$
for the zeros $\lambda$ of $G$.
Also $|G_1(iy)/S_2(iy)| \asymp 1$, so 
\begin{equation}
\label{imp2}
\bigg|\frac{G(iy)}{A(iy)}\bigg|\asymp 
\bigg|\frac{S(iy)}{A(iy)}\bigg| \asymp |y|^{-1}, \qquad |y| 
\to \infty,
\end{equation}
and by \cite[Theorem 26]{br}, $G\in\he +z\he$.
Estimate (\ref{imp2}) also 
yields the interpolation formula (\ref{ext2a}).

It remains to show that $G$ is the generating function of a complete and minimal 
system of kernels such that its biorthogonal is also complete.

To prove the first statement, we use that, 
by the construction, $S/G = T_0 S_2/(\tilde T_0 G_1)$ is a Smirnov 
class function both in the upper and the lower half-planes, while $A/S$ 
is a Herglotz function and, thus, also a Smirnov class function.
Hence, if $GH \in \mathcal{H}(E)$, then an application of
Krein's theorem (Subsection~\ref{51}) yields that $H$ is of zero exponential type.
Then it follows from \eqref{imp2} that $H$ is a polynomial, which contradicts the fact
that $G \notin \he$.

Finally, by \eqref{imp3}, $|c_n| \gtrsim |t_n|^{-N-1}|a_n|$,
thus $\mu_n \gtrsim |t_n|^{-M}$ and also $\sum_{n} \mu_n =\infty$.
Then, by \cite[Theorem 1.2]{bb}, the system biorthogonal to
$\{K_\lambda: \ G(\lambda)=0\}$ is also complete. 
This completes the construction of the example 
(and, thus, the proof of Theorem \ref{example}).

\end{document}